\documentclass[10pt]{article}

\usepackage{url}
\usepackage{amssymb,enumerate}
\usepackage{amsmath,amsfonts}
\usepackage{amsthm}
\usepackage{latexsym}
\usepackage{mathrsfs}
\usepackage{color}
\usepackage{microtype}
\usepackage{geometry}

\usepackage{algorithm}
\usepackage{algorithmicx}
\usepackage{algpseudocode}
\usepackage{xcolor}

\usepackage{multirow}
\allowdisplaybreaks[4]

\usepackage{amsmath,amsfonts,bm}

\def\1{\bm{1}}

\DeclareMathAlphabet{\mathsfit}{\encodingdefault}{\sfdefault}{m}{sl}
\SetMathAlphabet{\mathsfit}{bold}{\encodingdefault}{\sfdefault}{bx}{n}

\newcommand{\E}{\mathbb{E}}

\newcommand{\R}{\mathbb{R}}

\DeclareMathOperator*{\argmin}{arg\,min}

\theoremstyle{plain}
\newtheorem{theorem}{Theorem}

\newtheorem{lemma}{Lemma}

\newtheorem{assumption}{Assumption}

\newtheorem{definition}{Definition}

\newtheorem{remark}{Remark}

\usepackage{bbm}
\usepackage{graphicx, color}
\usepackage{algorithm, algpseudocode} 
\usepackage{mathrsfs} 

\usepackage{lipsum}

\title{Accelerated stochastic first-order method for convex optimization under heavy-tailed noise}
\author{Chuan He\thanks{Department of Mathematics, Link\"oping University, Sweden (email: {\tt chuan.he@liu.se}, {\tt bowen.a.li@liu.se}). This work was partially supported by the Wallenberg AI, Autonomous Systems and Software Program (WASP) funded by the Knut and Alice Wallenberg Foundation.}
\and
Bowen Li$^*$
\and
Zhaosong Lu\thanks{Department of Industrial and Systems Engineering, University of Minnesota, USA (email: {\tt zhaosong@umn.edu}). This work was partially supported by the Office of Naval Research under Award N00014-24-1-2702, the Air Force Office of Scientific Research under Award FA9550-24-1-0343, and the National Science Foundation under Awards 2211491 and 2435911.}
}

\date{
September 20, 2026
}

\begin{document}
\maketitle
	
\begin{abstract}
We study convex composite optimization problems, where the objective function is given by the sum of a prox-friendly function and a convex function whose subgradients are estimated under heavy-tailed noise. Existing work often employs gradient clipping or normalization techniques in stochastic first-order methods to address heavy-tailed noise. 
In this paper, we analyze the first-order oracle complexity of vanilla stochastic algorithms---without additional modifications such as clipping or normalization---for solving these problems. 
In particular, we establish that an accelerated stochastic proximal subgradient method achieves a first-order oracle complexity for finding an approximate optimal solution in expectation that is universally optimal for smooth, weakly smooth, and nonsmooth convex optimization, as well as for stochastic convex optimization under heavy-tailed noise. Moreover, we derive high-probability first-order oracle complexity bounds for the accelerated stochastic proximal subgradient method under heavy-tailed and sub-Weibull noise, respectively. Numerical experiments are further provided to illustrate the numerical behavior of the methods.
\end{abstract}

{\small \noindent\textbf{Keywords:} Convex composite optimization, heavy-tailed noise, accelerated stochastic proximal subgradient method, first-order oracle complexity}

\medskip

{\small \noindent\textbf{Mathematics Subject Classification:} 49M37, 90C15, 90C25, 90C30, 90C90}

\section{Introduction}
In this paper, we consider a class of convex composite optimization problems of the form
\begin{align}\label{ucpb}
F^* := \min_{x\in\R^n} \{F(x) := f(x) + h(x)\},    
\end{align}
where $f,h:\R^n\to(-\infty,\infty]$ are proper lower semicontinuous convex functions such that $\mathrm{dom}\,h\subseteq\mathrm{dom}\,f$. We assume that $f$ satisfies a hybrid of smooth and nonsmooth conditions: 
\begin{equation} \label{cond:hybd}
\|f^\prime(y)-f^\prime(x)\|\le L_f\|y-x\| + H_f\|y-x\|^\nu + M_f \ \qquad \forall f^\prime(y)\in \partial f(y), f^\prime(x)\in\partial f(x), x,y\in\mathrm{dom}\,f
\end{equation}
for some constants $L_f,H_f,M_f\ge0$ and $\nu\in(0,1)$.  In addition, we assume that the proximal operator associated with $h$ can be computed exactly. Clearly,  this class of functions $f$ includes Lipschitz smooth, H\"{o}lder smooth, and Lipschitz continuous functions, as well as any nonnegative combination of functions from these three subclasses. As recently observed in \cite[Example 1]{nesterov2025universal}, the sum of a Lipschitz smooth function and a H\"{o}lder smooth function is not necessarily a H\"{o}lder smooth function. Moreover, the sum of  a H\"{o}lder smooth function and a Lipschitz continuous function is not necessarily a Lipschitz continuous function. Consequently, the class of problems under consideration is broader than the class of problems studied in \cite{nesterov2015universal} with $f$ satisfying
\[
\|f^\prime(y)-f^\prime(x)\|\le H_f\|y-x\|^\nu\ \quad \forall f^\prime(y)\in \partial f(y), f^\prime(x)\in\partial f(x), x,y\in\mathrm{dom}\,f
\]
for some $H_f>0$ and $\nu\in [0,1]$.

With the rise of data science, instances of problem \eqref{ucpb} are increasingly common in modern, often large-scale, applications. As a result, the subgradients of $f$ are typically expensive to compute exactly and can usually only be approximated using stochastic estimators. Stochastic first-order methods have been extensively studied for solving \eqref{ucpb} and its variants; see, e.g., \cite{alacaoglu2025towards,bottou2018optimization,davis2021low,foster2019complexity,lan2012optimal,liang2024single,liu2024high,moulines2011non,nemirovski2009robust,nemirovski1983problem,polyak1992acceleration,robbins1951stochastic,shalev2009stochastic}. Remarkably, an optimal method has been developed in \cite{lan2012optimal} for solving a special case of  \eqref{ucpb} with $H_f=0$ and $h$ being the indicator function of a simple closed convex set, under the assumption that the stochastic subgradient estimator $G(\cdot;\xi)$ of $f(\cdot)$ is unbiased and has bounded variance---that is, $G(\cdot;\xi)$ satisfies the following conditions:
\begin{align}\label{asp:class}
\E[G(x;\xi)] \in\partial f(x),\quad \E\big[\|G(x;\xi) -\E[G(x;\xi)]\|^2\big] \le \sigma^2\quad \forall x\in\R^n
\end{align}
for some $\sigma>0$. Under these conditions, it has been shown in \cite{lan2012optimal} that a projected stochastic subgradient method with Nesterov's acceleration scheme achieves a first-order oracle complexity of 
\begin{align}\label{cplx-old}
\mathcal{O}\Big(\Big(\frac{L_f}{\epsilon}\Big)^{\frac{1}{2}} + \Big(\frac{M_f + \sigma}{\epsilon}\Big)^2\Big)\ \ \text{and}\ \ \mathcal{O}\Big(\Big(\frac{L_f}{\epsilon}\Big)^{\frac{1}{2}} + \Big(\frac{M_f + \sigma_0\log(1/\delta)}{\epsilon}\Big)^2\Big)    
\end{align}
for finding an $\epsilon$-stochastic optimal solution and an $(\epsilon,\delta)$-stochastic optimal solution, respectively,\footnote{An $\epsilon$-stochastic optimal solution refers to an $\epsilon$-optimal solution in expectation, while an $(\epsilon,\delta)$-stochastic optimal solution refers to an $\epsilon$-optimal solution with probability at least $1-\delta$ (see Definition~\ref{def:apx-sol} for details).}
where the latter complexity bound requires an additional sub-Gaussian noise assumption (\cite[Assumption A2]{lan2012optimal}):
\begin{align}\label{asp:sub-Gauss}
\E[\exp\{\|G(x;\xi) - \E[G(x;\xi)]\|^2/\sigma_0^2\}]\le e\qquad \forall x\in\R^n
\end{align}
for some $\sigma_0>0$. The first complexity bound {in \eqref{cplx-old}} recovers the optimal results achieved by first-order methods for smooth, nonsmooth, and stochastic convex optimization in a unified manner.

In recent years, with the development of machine learning, {communication systems, differential privacy,} and related fields, challenging stochastic optimization problems often extend beyond those satisfying classical assumptions imposed in \eqref{asp:class} {and \eqref{asp:sub-Gauss}}. 
{Recent studies have observed that the noise encountered in practical applications can exhibit heavy-tailed behavior. Such noise may arise from stochastic sampling in neural network training \cite{gurbuzbalaban2021heavy,simsekli2019heavy,simsekli2019tail,zhang2020adaptive}, communication noise in distributed systems \cite{gulgun2023massive,zhang2021measurements}, and deliberate noise injection for differential privacy \cite{simsekli2024differential}. These observations motivate the study of stochastic algorithms under heavy-tailed noise conditions:}
\begin{align*}
\E[G(x;\xi)] \in\partial f(x),\quad \E\big[\|G(x;\xi) -\E[G(x;\xi)]\|^\alpha\big] \le \sigma^\alpha\qquad \forall x\in\R^n
\end{align*}
for some $\sigma>0$ and $\alpha\in(1,2]$, generalizing the classical assumptions in \eqref{asp:class}.
Indeed, when $\alpha<2$, gradient estimators $G(\cdot;\xi)$ can exhibit unbounded variance, which may preclude the applicability of many classic algorithmic frameworks for them that are specifically developed for problems under condition \eqref{asp:class}. Notably, most existing algorithmic developments in stochastic optimization under heavy-tailed noise rely on gradient clipping  \cite{cutkosky2021high,gorbunov2020stochastic,liu2023stochastic,nguyen2023improved,sadiev2023high,zhang2020adaptive} or normalization techniques \cite{he2025complexity,hubler2024gradient,liu2025nonconvex,sun2024gradient}, providing theoretical justification for their empirical success. 
Nevertheless, a recent study \cite{fatkhullin2025can} shows that vanilla SGD, without using gradient clipping or gradient normalization, can be applied to a special case of \eqref{ucpb} with $H_f=0$ and $h$ being the indicator function of a simple compact convex set, achieving a first-order oracle complexity of $\mathcal{O}(\epsilon^{-\alpha/(\alpha-1)})$ for finding an $\epsilon$-stochastic optimal solution of \eqref{ucpb}. Similar complexity bounds for online convex optimization are established in \cite{liu2026online}.  Given that no acceleration scheme is used in {\cite{fatkhullin2025can,liu2026online}}, the following natural question arises:

{\quote \emph{Is an accelerated vanilla stochastic algorithm without clipping or normalization applicable to the general problem  \eqref{ucpb} under heavy-tailed noise?}}

\medskip

This paper provides an affirmative answer to this question. {Specifically, we establish first-order oracle complexity guarantees for a stochastic proximal subgradient method (SPGM) and its accelerated variant for solving problem \eqref{ucpb} under heavy-tailed and sub-Weibull noise, where the latter subsumes the classical sub-Gaussian noise conditions.} Our main contributions are summarized below.

\begin{itemize}
\item We show that an SPGM and its accelerated counterpart, without any modifications such as clipping or normalization, can find an approximate optimal solution of \eqref{ucpb} both in expectation and with high probability. 

\item {Under the heavy-tailed noise condition (Assumption \ref{asp:basic}(c)), we} show that the SPGM (Algorithm~\ref{alg:dq-f}) achieves a first-order oracle complexity of 
\begin{subequations}\label{cplx-1}
\begin{align}
&\mathcal{O}\Big(\frac{L_f{D_h^2}}{\epsilon}+\Big(\frac{H_f{D_h^{1+\nu}}}{\epsilon}\Big)^{\frac{2}{1+\nu}} + \Big(\frac{M_f{D_h}}{\epsilon}\Big)^2 + \Big(\frac{\sigma{D_h}}{\epsilon}\Big)^{\frac{\alpha}{\alpha-1}}\Big),\label{cplx-1-e}\\
\text{and}\ \ &{\mathcal{O}}\Big(\frac{L_f{D_h^2}}{\epsilon} + \Big(\frac{H_f{D_h^{1+\nu}}}{\epsilon}\Big)^{\frac{2}{1+\nu}} + \Big(\frac{M_f{D_h}}{\epsilon}\Big)^2 + \Big(\frac{\sigma{D_h(1/\delta)^{\frac{1}{\alpha}}}}{\epsilon}\Big)^{\frac{\alpha}{\alpha-1}}\Big)\label{cplx-1-p}
\end{align}
\end{subequations}
for finding an $\epsilon$-stochastic optimal solution of \eqref{ucpb}, and an $(\epsilon,\delta)$-stochastic optimal solution, respectively, {where $D_h$ is the diameter of $\mathrm{dom}\,h$ defined in \eqref{def:diam-h}.} In addition, we establish that the accelerated SPGM (Algorithm \ref{alg:ac-dq-f}) achieves a first-order oracle complexity of 
\begin{subequations}\label{cplx-a2}
\begin{align}
&\mathcal{O}\Big(\Big(\frac{L_f{D_h^2}}{\epsilon}\Big)^{\frac{1}{2}} + \Big(\frac{H_f{D_h^{1+\nu}}}{\epsilon}\Big)^{\frac{2}{1+3\nu}} + \Big(\frac{M_f{D_h}}{\epsilon}\Big)^2 + \Big(\frac{\sigma{D_h}}{\epsilon}\Big)^{\frac{\alpha}{\alpha-1}}\Big),\label{cplx-a2-e}\\
\text{and}\ \ &{\mathcal{O}}\Big(\Big(\frac{L_f{D_h^2}}{\epsilon}\Big)^{\frac{1}{2}} +\Big(\frac{H_f{D_h^{1+\nu}}}{\epsilon}\Big)^{\frac{2}{1+3\nu}} + \Big(\frac{M_f{D_h}}{\epsilon}\Big)^2 + \Big(\frac{\sigma{D_h(1/\delta)^{\frac{1}{\alpha}}}}{\epsilon}\Big)^{\frac{\alpha}{\alpha-1}}\Big) \label{cplx-2-p}
\end{align}    
\end{subequations}
for finding an $\epsilon$-stochastic optimal solution of \eqref{ucpb}, and an $(\epsilon,\delta)$-stochastic optimal solution, respectively.

\item {Under the sub-Weibull noise condition (Assumption \ref{asp:sub-weibull}), we establish that the SPGM and the accelerated SPGM achieve a first-order oracle complexity of 
\begin{subequations}
\begin{align}
&\mathcal{O}\Big(\frac{L_fD_h^2}{\epsilon} + \Big(\frac{H_fD_h^{1+\nu}}{\epsilon}\Big)^{\frac{2}{1+\nu}} + \Big(\frac{M_fD_h}{\epsilon}\Big)^2 + \Big(\frac{\sigma_0D_h\ln(1/\delta)^{\frac{1}{2}}}{\epsilon}\Big)^{2} + \Big(\frac{\sigma_0D_h\ln(1/\delta)^{\frac{1}{\beta}}}{\epsilon}\Big)^{\max\{\beta,1\}}\Big),\label{subwei-cmplx-spg}\\
\text{and}\ \ &\mathcal{O}\Big(\Big(\frac{L_fD_h^2}{\epsilon}\Big)^{\frac{1}{2}} +\Big(\frac{H_fD_h^{1+\nu}}{\epsilon}\Big)^{\frac{2}{1+3\nu}} + \Big(\frac{M_fD_h}{\epsilon}\Big)^2 + \Big(\frac{\sigma_0D_h\ln(1/\delta)^{\frac{1}{2}}}{\epsilon}\Big)^{2} + \Big(\frac{\sigma_0D_h\ln(1/\delta)^{\frac{1}{\beta}}}{\epsilon}\Big)^{\max\{\beta,1\}}\Big)\label{subwei-cmplx-aspg}
\end{align}
\end{subequations}
for finding an $(\epsilon,\delta)$-stochastic optimal solution, respectively, where $\sigma_0$ and $\beta$ are given in Assumption \ref{asp:sub-weibull}.
}
\end{itemize}

{
To the best of our knowledge, the unified complexity bounds in \eqref{cplx-a2} and \eqref{subwei-cmplx-aspg} for accelerated SPGM are new in the general setting considered here. Moreover, the bounds in \eqref{cplx-a2-e} and \eqref{subwei-cmplx-aspg} match or nearly match existing lower bounds for smooth, weakly smooth, nonsmooth, and stochastic convex optimization under the noise settings specified below.

\begin{itemize}
\item When the noise vanishes (i.e., $\sigma=0$ and $\sigma_0=0$), Algorithm \ref{alg:ac-dq-f} reduces to a deterministic accelerated SPGM. In this case, our complexity bounds in \eqref{cplx-a2} and \eqref{subwei-cmplx-aspg} match the lower complexity bounds presented in Table~\ref{table:lcb1} for smooth, weakly smooth, and nonsmooth convex optimization. Hence, our complexity bounds are optimal for all three classes of problems.

\item Under only heavy-tailed noise, the leading stochastic-noise term in the complexity bound \eqref{cplx-a2-e}, namely, the last term, matches the lower bound for finding an $\epsilon$-stochastic optimal solution (see the first column of Table~\ref{table:lcb2}). The last stochastic-noise term in \eqref{cplx-2-p} has a polynomial dependence on $\epsilon$ and $\delta$ similar to that of a related lower bound established in \cite[Theorem~3.3]{fatkhullin2025can}. Whether this dependence is sharp in our setting remains unknown.\footnote{{ A related lower complexity bound has been established in \cite[Theorem~3.3]{fatkhullin2025can} for a class of stochastic gradient methods to find an $x$
satisfying $F(x)-F^*\le\epsilon$ with probability at least $1-\delta$. However, this lower-bound construction applies only to unconstrained problems under a bounded gradient-moment condition. Thus, it does not directly establish the optimality of our upper bound \eqref{cplx-2-p} for bounded-domain composite optimization under a bounded noise-moment condition.
}} Moreover, under the additional sub-Weibull noise assumption, the last two terms in our complexity bound \eqref{subwei-cmplx-aspg}, which account for stochastic noise, match, up to factors polynomial in $\ln(1/\delta)$, the lower bound for finding an $(\epsilon,\delta)$-stochastic optimal solution under sub-Gaussian noise (see the last column of Table~\ref{table:lcb2}). When $\beta=2$, these terms improve the dependence on $\ln(1/\delta)$ compared with the bound in \cite{lan2012optimal}, while matching the corresponding bound in \cite{liu2023high} under the sub-Gaussian noise assumption. 
\end{itemize}
}

\begin{table}[tbhp]
\renewcommand{\arraystretch}{1.5}
\centering
{
\label{table:lcb1}
\caption{Lower complexity bounds for deterministic proximal subgradient method for solution type $F(x)-F^*\le\epsilon$}
\resizebox{\textwidth}{!}{
\begin{tabular}{l||c|c|c}
\hline
Problem class & $L_f$-smooth & $(H_f,\nu)$-weakly smooth & $M_f$-nonsmooth \\
\hline
Lower bound & $\displaystyle\Omega\Big(\Big(\frac{L_fD_h^2}{\epsilon}\Big)^{\frac{1}{2}}\Big)$ \cite[Corollary 1]{guzman2015lower} & $\displaystyle\Omega\Big(\Big(\frac{H_fD_h^{1+\nu}}{\epsilon}\Big)^{\frac{2}{1+3\nu}}\Big)$ \cite[Corollary 1]{guzman2015lower} & $\displaystyle\Omega\Big(\Big(\frac{M_fD_h}{\epsilon}\Big)^2\Big)$ \cite[Theorem 3.13]{bubeck2015convex} \\
\hline 
\end{tabular}
}
}
\end{table}

\begin{table}[tbhp]
\centering
{
\caption{Lower complexity bounds for stochastic proximal subgradient method} 
\label{table:lcb2}
\renewcommand{\arraystretch}{1.5}
\begin{tabular}{l||c|c}
\hline
Noise type & {$(\sigma,\alpha)$-bounded moment} & $\sigma_0$-sub-Gaussian \\
\hline
Solution type & $\E[F(x)-F^*]\le\epsilon$ & $F(x)-F^*\le\epsilon$ w.p. $1-\delta$ \\
Lower bound & $\displaystyle\Omega\Big(\Big(\frac{\sigma D_h}{\epsilon}\Big)^{\frac{\alpha}{\alpha-1}}\Big)$ \cite[Theorem 9]{vural2022mirror} & $\displaystyle\Omega\Big(\Big(\frac{\sigma_0 D_h\ln(1/\delta)^{\frac{1}{2}}}{\epsilon}\Big)^{2}\Big)$ \cite[Proposition 1]{pmlr-v247-carmon24a} \\
\hline 
\end{tabular}
}
\end{table}

{Based on the above discussion, we can draw the following conclusions.
\begin{itemize}
\item[(i)] By combining the lower complexity bounds in Table~\ref{table:lcb1} and the first column of Table~\ref{table:lcb2}, we obtain the following unified lower complexity bound for finding an $\epsilon$-stochastic optimal solution of problem~\eqref{ucpb} using a first-order method under only heavy-tailed noise:
 \[
 \Omega\bigg(\max\bigg\{\Big(\frac{L_f D_h^2}{\epsilon}\Big)^{\frac{1}{2}}, \Big(\frac{H_f{D_h^{1+\nu}}}{\epsilon}\Big)^{\frac{2}{1+3\nu}},\Big(\frac{M_f D_h}{\epsilon}\Big)^2,\Big(\frac{\sigma D_h}{\epsilon}\Big)^{\frac{\alpha}{\alpha-1}}\bigg\}\bigg).
 \]
Hence, the complexity bound in \eqref{cplx-a2-e} achieved by accelerated SPGM is universally optimal for finding an $\epsilon$-stochastic optimal solution of \eqref{ucpb} under only heavy-tailed noise.
\item[(ii)] By combining the lower complexity bounds in Table~\ref{table:lcb1} and the last column of Table~\ref{table:lcb2}, we obtain the following unified lower complexity bound for finding an $(\epsilon,\delta)$-stochastic optimal solution of problem~\eqref{ucpb} using a first-order method under sub-Gaussian noise:
 \[
  \Omega\bigg(\max\bigg\{\Big(\frac{L_f D_h^2}{\epsilon}\Big)^{\frac{1}{2}}, \Big(\frac{H_f{D_h^{1+\nu}}}{\epsilon}\Big)^{\frac{2}{1+3\nu}},\Big(\frac{M_f D_h}{\epsilon}\Big)^2,\Big(\frac{\sigma_0 D_h\ln(1/\delta)^{\frac{1}{2}}}{\epsilon}\Big)^{2}\bigg\}\bigg).
 \]
Hence, the complexity bound in \eqref{subwei-cmplx-aspg} achieved by accelerated SPGM is near-optimal, up to factors polynomial in $\ln(1/\delta)$, for finding an $(\epsilon,\delta)$-stochastic optimal solution of \eqref{ucpb} under sub-Weibull noise.
\end{itemize}
}

The rest of this paper is organized as follows. Section~\ref{sec:notation} presents notation and assumptions. In Sections~\ref{sec:psg} and \ref{sec:apsg}, we present SPGM and its accelerated counterpart along with their first-order oracle complexity results for finding an approximate solution of problem \eqref{ucpb} under heavy-tailed noise.    Section~\ref{sec:experiment} presents  preliminary numerical results illustrating the performance of the proposed methods. We provide the proof of the main results in Section \ref{sec:proof} {and discuss directions for future research in Section \ref{sec:fw}.}

\subsection{Notation and assumptions} \label{sec:notation}

Throughout this paper, we use $\R^n$ to stand for the $n$-dimensional Euclidean space, and $\|\cdot\|$ to denote the Euclidean norm for vectors. {Let $\mathbb{Z}_+^m$ denote the set of $m$-dimensional vectors with nonnegative integer entries.} For any proper closed convex function $\varphi$, we denote its subdifferential by $\partial \varphi$ and define the proximal mapping associated with $\varphi$, with parameter $\eta > 0$, as
\begin{align*}
\mathrm{prox}_{\eta \varphi}(x):=\argmin_{z\in\R^n}\Big\{\varphi(z) + \frac{1}{2\eta}\|z-x\|^2\Big\}.    
\end{align*}
We denote the domain of $\varphi$ as $\mathrm{dom}\,\varphi$. For any $\mathcal{A}\subseteq\mathbb{R}^m$ and $u\in\mathbb{R}^m$, we define the Boolean indicator function $\mathbbm{1}_{\mathcal{A}}(u)$ to be $1$ if $u\in\mathcal{A}$ and $0$ otherwise, {and the extended-valued indicator function $\iota_{\mathcal{A}}(u)$ to be $0$ if $u\in\mathcal{A}$ and $+\infty$ otherwise.} 
{For a given finite set $\mathcal{A}$, we let $|\mathcal{A}|$ denote its cardinality. Let the Gamma function be defined by $\Gamma(z) := \int^\infty_0 t^{z-1} e^{-t}\mathrm{d}t$ for any $z>0$.}
In addition, we use $\mathcal{O}(\cdot)$ {and $\Omega(\cdot)$} to denote the standard big-O {and big-Omega} notation, {respectively}.

We now make the following assumption throughout this paper.

\begin{assumption}\label{asp:basic}
\begin{enumerate}[{\rm (a)}]
\item  The function $f$ satisfies \eqref{cond:hybd} for some constants $L_f,H_f,M_f\ge0$ and $\nu\in(0,1)$. 
\item The proximal operator associated with $h$ can be exactly evaluated, and its domain $\mathrm{dom}\,h$ is bounded. 
\item The stochastic subgradient estimator $G:\R^n\times\Xi\to\R^n$ satisfies
\begin{align}\label{cond:ac}
\E[G(x;\xi)] \in \partial f(x),\quad \E[\|G(x;\xi) - \E[G(x;\xi)]\|^\alpha] \le \sigma^\alpha \quad\forall x\in\mathrm{dom}\,f
\end{align}
for some $\sigma>0$ and $\alpha\in(1,2]$.
\end{enumerate}
\end{assumption}

We next make some remarks on Assumption \ref{asp:basic}.

\begin{remark}
(i) The class of $f$ satisfying Assumption~\ref{asp:basic}(a) is broad, which includes smooth (gradient Lipschitz continuous), weakly smooth (gradient H\"{o}lder continuous), and nonsmooth (Lipschitz continuous) functions, as well as any nonnegative combination of functions from these subclasses. Problems of the form \eqref{ucpb} with $f$ from these subclasses have been extensively studied in the literature (e.g., \cite{grimmer2024optimal,guigues2024universal,lan2012optimal}). However, apart from the recent work \cite{nesterov2025universal}, problem \eqref{ucpb} with $f$ satisfying Assumption~\ref{asp:basic}(a) has received limited attention. In particular,  \cite{nesterov2025universal} proposed first-order methods and established complexity guarantees for such problems under a deterministic first-order oracle, where the exact gradient or an exact subgradient of $f$ is used.

(ii) By a standard argument for deriving the descent inequality, \eqref{cond:hybd} implies
\begin{align}\label{ineq:desc}
f(y) \le f(x) + f^\prime(x)^T(y-x) + \frac{L_f}{2}\|y-x\|^2 + \frac{H_f}{1+\nu}\|y-x\|^{1+\nu} + M_f\|y-x\|
\end{align}
holds for all $f^\prime(x)\in\partial f(x), x,y\in\mathrm{dom}\,f$. It follows from \cite[Lemma 2]{nesterov2015universal} that\footnote{{This inequality is obtained from \cite[Lemma 2]{nesterov2015universal} by replacing $(\delta,M_\nu)$ with $(\varepsilon/4,H_f)$.}}
\[
 \frac{H_f}{1+\nu}\|y-x\|^{1+\nu}  \leq \frac12 L(\varepsilon) \|y-x\|^2 + \frac{\varepsilon}{8},
\]
where
\begin{align}\label{def:inexact-Lip}
L(\varepsilon) := H_f^{\frac{2}{1+\nu}}\Big(\frac{4}{\varepsilon}\Big)^{\frac{1-\nu}{1+\nu}}\qquad\forall \varepsilon>0.
\end{align}
This together with \eqref{ineq:desc} implies that
\begin{align}\label{ineq:desc_new}
f(y) \le f(x) + f^\prime(x)^T(y-x) + \frac{1}{2}\big(L_f+L(\varepsilon) \big)\|y-x\|^2 + M_f\|y-x\|+ \frac{\varepsilon}{8} 
\end{align}
holds for any $\varepsilon>0$ and all $f^\prime(x)\in\partial f(x), x,y\in\mathrm{dom}\,f$.

(iii) {
Assumption~\ref{asp:basic}(b) appears to be necessary for establishing complexity guarantees for SPGM and accelerated SPGM to find an \(\epsilon\)-stochastic optimal solution of \eqref{ucpb} under heavy-tailed noise. Without a bounded-domain assumption, vanilla stochastic gradient methods may generate iterates with infinite expected objective values under heavy-tailed noise, as illustrated by the lower-bound construction in \cite[Section~3.2]{fatkhullin2025can}. Moreover, a similar assumption is also commonly imposed in the literature for the special case where $h$ is the indicator function of a simple closed convex set; see, e.g., \cite{fatkhullin2025can,lan2012optimal,liu2026online}.} Since $\mathrm{dom}\,h$ is assumed to be bounded, its diameter, defined as
\begin{align}\label{def:diam-h}
D_h := \sup_{x,y\in\mathrm{dom}\,h}\{\|x-y\|\},   
\end{align}
is finite. Assumption~\ref{asp:basic}(c) states that $G(x;\xi)$ is an unbiased estimator of a subgradient of $f(x)$ and that its $\alpha$th central moment is uniformly bounded. This assumption is weaker than the commonly used bounded-variance assumption, which corresponds to the case $\alpha=2$. When $\alpha\in(1,2)$, it permits heavy-tailed stochastic subgradient noise (see, e.g., \cite{zhang2020adaptive}), a phenomenon commonly encountered in machine learning {and other} applications.  
\end{remark}

We next introduce another assumption, which will be used to establish {improved} complexity bounds for finding approximate optimal solutions of problem \eqref{ucpb} with high-probability guarantees.

\begin{assumption}\label{asp:sub-weibull}
The stochastic subgradient estimator $G:\R^n\times\Xi\to\R^n$ satisfies
\begin{align}\label{cond:sub-wb}
\E[\exp\{\|G(x;\xi) - \E[G(x;\xi)]\|^{\beta}/{{\sigma_0^{\beta}}}\}]\le e\qquad \forall x\in\mathrm{dom}\,f
\end{align}
for some ${\sigma_0}>0$ and {$\beta\in(0,2]$}.
\end{assumption}

We now make some remarks regarding Assumption \ref{asp:sub-weibull}.


\begin{remark}
Assumption~\ref{asp:sub-weibull} states that the stochastic subgradient noise follows a sub-Weibull distribution (see, e.g., \cite{vladimirova2020sub}). This assumption generalizes the sub-Gaussian condition in \eqref{asp:sub-Gauss}, which corresponds to the special case $\beta=2$. {However, it is stronger than Assumption~\ref{asp:basic}(c). We impose this assumption to establish high-probability guarantees with a polylogarithmic dependence on $1/\delta$. Complexity bounds for stochastic first-order methods under sub-Weibull noise have received considerable attention in recent years; see, e.g., \cite{attia2024general,li2022high,liu2024revisiting,madden2024high}.}
\end{remark}

We next introduce two types of approximate stochastic optimal solutions of problem \eqref{ucpb} considered in this paper.

\begin{definition}\label{def:apx-sol}
Let $\epsilon>0$ and $\delta\in(0,1)$ be given. We say that 
\begin{itemize}
    \item $x\in\R^n$ is {\it an $\epsilon$-stochastic optimal solution} of \eqref{ucpb} if it satisfies $\E[F(x) - F^*]\le\epsilon$; and
    \item $x\in\R^n$ is {\it an $(\epsilon,\delta)$-stochastic optimal solution} of \eqref{ucpb} if it satisfies $F(x) - F^*\le\epsilon$ with probability at least $1-\delta$.
\end{itemize}
\end{definition}

\section{A stochastic proximal subgradient method}\label{sec:psg}

In this section, we present an SPGM and establish its first-order oracle complexity for solving problem \eqref{ucpb} under heavy-tailed and sub-Weibull noise. 

The SPGM was originally proposed for solving a special case of \eqref{ucpb} with $H_f = 0$ under the conditions \eqref{asp:class} (see, e.g., \cite{ lan2012optimal,nemirovski2009robust}). We now extend it to the general problem \eqref{ucpb} in the presence of heavy-tailed noise. Specifically, the SPGM generates two sequences, $\{x^k\}$ and $\{z^k\}$. At each iteration $k \ge 0$, it first updates $x^{k+1}$ by performing a stochastic proximal subgradient step and then computes $z^{k+1}$ as a weighted average of the past iterates $\{x^t\}_{t=1}^{k+1}$. The details of this method are presented in Algorithm~\ref{alg:dq-f}, with specific choices of step sizes provided in Theorem~\ref{thm:cv-psg}.

\begin{algorithm}[!htbp]
\caption{A stochastic proximal subgradient method}
\label{alg:dq-f}
\begin{algorithmic}[0]
\State \textbf{Input:} starting point $x^0\in\mathrm{dom}\,h$, step sizes $\{\eta_k\}\subset(0,\infty)$.
\For{$k=0,1,2,\ldots$}
\State Randomly draw a sample $\xi_k$, and update the next iterate:
\begin{align}\label{prox-step}
x^{k+1} = \mathrm{prox}_{\eta_k h}(x^k - \eta_k G(x^k;\xi_k)).
\end{align}
\State Compute the weighted average:
\begin{align}\label{weight-average}
z^{k+1} = \Big(\sum_{t=0}^k \eta_t \Big)^{-1}\sum_{t=0}^k \eta_t x^{t+1}.
\end{align}
\EndFor
\end{algorithmic}
\end{algorithm}

To analyze Algorithm~\ref{alg:dq-f}, we introduce several quantities below. In particular, let $\alpha\in(1,2]$, $\beta\in(0,2]$, and $\sigma,\sigma_0>0$ be as specified in Assumptions~\ref{asp:basic} and \ref{asp:sub-weibull}, respectively. For any $\varepsilon>0$ and $\delta\in (0,1)$, we define 
{
\begin{align}
&\sigma_1^2 :=\frac{2e\Gamma(2/\beta)\sigma_0^2}{\beta},\quad \sigma_2^2:=\bigg[1 + \bigg(\frac{2e\Gamma(2/\beta)}{\beta}\bigg)^{\frac{\beta}{2}}\bigg]^{\frac{2}{\beta}}\sigma_0^2,\label{def:sigma12}\\
&\Lambda(\varepsilon)^2 := 8(\alpha-1)^2\Big(\frac{\sigma}{\alpha}\Big)^{\frac{\alpha}{\alpha-1}}\Big(\frac{8D_h}{\varepsilon}\Big)^{\frac{2-\alpha}{\alpha-1}},\quad \bar\Lambda(\delta,\varepsilon)^2:=\Big(\frac{2}{\delta}\Big)^{\frac{1}{\alpha-1}}\Lambda(\varepsilon)^2,\label{def:sigma1}\\
&\Lambda_1(\delta)^2:= e^{\frac{3}{2}}\bigg(\frac{8}{\beta}\bigg)^{\frac{2}{\beta}}(4\sigma_2D_h)^2\Big(\ln\Big(\frac{2}{\delta}\Big)+1\Big)^{\frac{1}{2}},\quad \Lambda_{2}(\delta)^2:= e^{\frac{3}{2}}\bigg(\frac{8}{\beta}\bigg)^{\frac{2}{\beta}}(4\sigma_2D_h)^2\Big(\ln\Big(\frac{2}{\delta}\Big)+1\Big)^{\frac{2}{\beta}},\label{def:C32}\\
&\bar{B}(\delta):= 2^{\frac{2-\alpha}{\alpha}}\sigma D_h\Big(\frac{2}{\delta}\Big)^{\frac{1}{\alpha}},\quad B_{1}(\delta):= 4e^{\frac{3}{2}}\Big(\frac{4}{\beta}\Big)^{\frac{1}{\beta}}\sigma_0 D_h\Big(\ln\Big(\frac{2}{\delta}\Big)+1\Big)^{\frac{1}{2}},\quad B_{2}(\delta):= 4e^{\frac{3}{2}}\Big(\frac{4}{\beta}\Big)^{\frac{1}{\beta}}\sigma_0 D_h\Big(\ln\Big(\frac{2}{\delta}\Big)+1\Big)^{\frac{1}{\beta}}.\label{def:BC31}
\end{align}
}

The theorem below establishes a complexity bound for Algorithm \ref{alg:dq-f} to compute an $\epsilon$-stochastic optimal solution and an $(\epsilon,\delta)$-stochastic optimal solution of \eqref{ucpb}, respectively. Its proof is deferred to Section \ref{subsec:pf-psg}.


\begin{theorem}\label{thm:cv-psg}
Suppose that Assumption~\ref{asp:basic} holds. Let $\epsilon,\delta\in(0,1)$ be arbitrarily chosen, and let $K$ be a pre-chosen maximum iteration number for running Algorithm \ref{alg:dq-f}. Let $L(\cdot)$, $D_h$, {$\sigma_1$, $\Lambda(\cdot)$, $\bar\Lambda(\cdot,\cdot)$, $\Lambda_1(\cdot)$, and $\Lambda_2(\cdot)$} be defined in \eqref{def:inexact-Lip}, \eqref{def:diam-h}, {\eqref{def:sigma12}-\eqref{def:C32}}, respectively, {$\bar{B}(\cdot)$, $B_{1}(\cdot)$, $B_{2}(\cdot)$ be defined in \eqref{def:BC31},} $L_f,M_f$ be given in Assumption \ref{asp:basic}(a), and let
\begin{align}
\eta & = \min\bigg\{\frac{1}{4(L_f + L(\epsilon))},\frac{D_h}{[2K(M_f^2 + {\Lambda(\epsilon)^2})]^{1/2}}\bigg\},\label{def:eta_k-psg-c}\\
{\bar\eta} & = \min\bigg\{\frac{1}{4(L_f + L(\epsilon))},\frac{D_h}{[2K(M_f^2 +{{\bar\Lambda}(\delta,\epsilon)^2})]^{1/2}}\bigg\},\label{def:eta_k-psg-c-2}\\
{\tilde\eta} &  = {\min\bigg\{\frac{1}{4(L_f + L(\epsilon))},\frac{D_h}{\big[2K\big(M_f^2 + 2\sigma_1^2 + \frac{\Lambda_1(\delta)^2}{2D_h^2K^{1/2}}+\frac{\Lambda_2(\delta)^2}{2D_h^2K}\big]^{1/2}}\bigg\}.}\label{def:eta_k-psg-c-3}
\end{align}
Then the following statements hold.
\begin{enumerate}[{\rm (i)}]
\item Let $\{z^k\}$ be generated by Algorithm \ref{alg:dq-f} with $\eta_k\equiv \eta$ for all $k\ge0$. Then, $\E[F(z^K) - F^*]\le\epsilon$ for all $K$ satisfying
\begin{align} 
K\ge \max\bigg\{\frac{8D_h^2(L_f + L(\epsilon))}{\epsilon}, \frac{8D_h^2(M_f + {\Lambda(\epsilon)})^2}{\epsilon^2},1\bigg\}.  \label{K1}
\end{align}
\item {Let $\{z^k\}$ be generated by Algorithm \ref{alg:dq-f} with $\eta_k\equiv\bar\eta$ for all $k\ge0$. Then, with probability at least $1-\delta$, $F(z^K) - F^*\le\epsilon$ holds for all $K$ satisfying}
\begin{align}\label{K11}
{ K\ge \max\bigg\{\frac{8D_h^2(L_f + L(\epsilon))}{\epsilon}, \frac{32D_h^2(M_f+{\bar\Lambda}(\delta,\epsilon))^2}{\epsilon^2},{\Big(\frac{4\bar{B}(\delta)}{\epsilon}\Big)^{\frac{\alpha}{\alpha-1}}},1\bigg\}.}
\end{align} 
\item Suppose additionally that Assumption \ref{asp:sub-weibull} holds. Let $\{z^k\}$ be generated by Algorithm \ref{alg:dq-f} with $\eta_k\equiv \tilde\eta$ for all $k\ge0$. Then, with probability at least $1-\delta$, $F(z^K) - F^*\le\epsilon$ holds for all $K$ satisfying 
\begin{align}
&{
K\ge \max\bigg\{\frac{8D_h^2(L_f + L(\epsilon))}{\epsilon}, \frac{16[2D_h(M_f + \sigma_1)+B_{1}(\delta)]^2}{\epsilon^2},16\Big(\frac{\Lambda_{1}(\delta)}{\epsilon}\Big)^{\frac{4}{3}},\Big(\frac{8B_2(\delta)}{\epsilon}\Big)^{\max\{\beta,1\}},\frac{8\Lambda_2(\delta)}{\epsilon},1\bigg\}.}  \label{K2}
\end{align}
\end{enumerate}
\end{theorem}

\begin{remark}
{(i) In view of \eqref{def:inexact-Lip}, \eqref{def:sigma1}, \eqref{def:BC31}, and Theorem~\ref{thm:cv-psg}(i)-(ii), we observe that, under Assumption~\ref{asp:basic}, Algorithm~\ref{alg:dq-f} achieves the first-order oracle complexities in \eqref{cplx-1-e} and \eqref{cplx-1-p} for finding an $\epsilon$-stochastic optimal solution and an $(\epsilon,\delta)$-stochastic optimal solution of \eqref{ucpb}, respectively. In addition, by \eqref{def:C32}, \eqref{def:BC31}, and Theorem~\ref{thm:cv-psg}(iii), under the additional sub-Weibull noise condition, Algorithm~\ref{alg:dq-f} achieves the first-order oracle complexity in \eqref{subwei-cmplx-spg} for finding an $(\epsilon,\delta)$-stochastic optimal solution of \eqref{ucpb}. This complexity bound depends only polylogarithmically on $1/\delta$ and thus improves upon the bound in \eqref{cplx-1-p}.  Moreover, the complexity bounds in \eqref{cplx-1-e}, \eqref{cplx-1-p}, and \eqref{subwei-cmplx-spg} achieve the optimal dependence on $\epsilon$ for nonsmooth convex problems (i.e., $L_f=0$ and $H_f=0$) \cite{fatkhullin2025can,liu2023stochastic}. For smooth and weakly smooth convex problems, however, these bounds are not optimal. They will be improved by an accelerated SPGM developed in the next section.

(ii) The result in Theorem~\ref{thm:cv-psg}(ii) is established using a new concentration inequality in Lemma~\ref{lem:hts} for martingale difference sequences with bounded \(\alpha\)th moments. This inequality requires weaker tail assumptions than the concentration inequality developed in \cite[Appendix A]{liu2024revisiting}, which applies only to sub-Weibull random variables. Moreover, the result in Theorem~\ref{thm:cv-psg}(iii) is established using a new concentration inequality (see Lemma~\ref{lem:mds}) under sub-Weibull noise. This inequality is stronger than the one developed in \cite[Appendix A]{liu2024revisiting}. Consequently, it yields a better dependence on $\ln(1/\delta)$ in our complexity bound \eqref{subwei-cmplx-spg} than the corresponding bound in \cite[Section 6.3]{liu2024revisiting}.

(iii) It is worth mentioning that gradient-clipping methods can achieve high-probability complexity bounds with a polylogarithmic dependence on $1/\delta$ under only heavy-tailed noise; see, e.g., \cite{gorbunov2020stochastic,liu2023stochastic,sadiev2023high,zhang2022parameter}. However, for SPGM to achieve high-probability complexity bounds with a polylogarithmic dependence on $1/\delta$, a sub-Weibull noise assumption appears to be needed. For example, the sub-Gaussian assumption (i.e., the sub-Weibull assumption with exponent 2) was imposed in \cite{lan2012optimal,nemirovski2009robust} for proximal-gradient-type methods to obtain high-probability complexity bounds with a polylogarithmic dependence on $1/\delta$.
}
\end{remark}

\section{An accelerated stochastic proximal subgradient method}\label{sec:apsg} 

In this section, we present an accelerated SPGM and establish its first-order oracle complexity for solving problem \eqref{ucpb} under heavy-tailed and sub-Weibull noise.  

\begin{algorithm}[!htbp]
\caption{An accelerated stochastic proximal subgradient method}
\label{alg:ac-dq-f}
\begin{algorithmic}[0]
\State \textbf{Input:} starting point $x^0=z^0\in\mathrm{dom}\,h$, step sizes $\{\eta_k\}\subset(0,\infty)$, weighting parameters $\{\gamma_k\}\subset(0,1]$.
\For{$k=0,1,2,\ldots$}
\State Compute the intermediate point:
\begin{align}\label{mid-update}
y^k = (1-\gamma_k)z^k + \gamma_k x^k.
\end{align}
\State Randomly draw a sample $\xi_k$,  and update the next iterate:
\begin{align}\label{a-prox-step}
x^{k+1} = \mathrm{prox}_{\eta_k h}(x^k - \eta_k G(y^k;\xi_k)).
\end{align}
\State Compute the weighted average:
\begin{align}\label{ag-update}
z^{k+1} = (1-\gamma_k)z^k + \gamma_k x^{k+1}.
\end{align}
\EndFor
\end{algorithmic}
\end{algorithm}

The accelerated SPGM was originally proposed in \cite{lan2012optimal} for solving a special case of problem~\eqref{ucpb} with $H_f = 0$ under conditions~\eqref{asp:class}. We now extend it to handle the general problem~\eqref{ucpb} under heavy-tailed noise. The accelerated SPGM can be viewed as a stochastic variant of Nesterov’s accelerated proximal gradient method~\cite[Algorithm~1]{Tse08}, obtained by replacing the deterministic gradient with a stochastic subgradient. Specifically, the method generates three sequences, $\{x^k\}$, $\{y^k\}$, and $\{z^k\}$. The sequence $\{x^k\}$ represents the main iterates, updated via a proximal operator. The sequence $\{z^k\}$ denotes the aggregated iterates, where each $z^k$ is a weighted average of $\{x^t\}_{t=0}^k$. 
The sequence $\{y^k\}$ serves as an intermediate sequence, with each $y^k$ computed as an average of $x^k$ and $z^k$. The complete description of the method is given in Algorithm~\ref{alg:ac-dq-f}, and the specific choices of step sizes are provided in Theorem~\ref{thm:cv-apsg}.

The next theorem establishes a complexity bound for Algorithm \ref{alg:ac-dq-f} to compute an $\epsilon$-stochastic optimal solution and an $(\epsilon,\delta)$-stochastic optimal solution of \eqref{ucpb}, respectively. Its proof is deferred to Section \ref{subsec:pf-apsg}.

\begin{theorem}\label{thm:cv-apsg}
Suppose that Assumption~\ref{asp:basic} holds.  
Let $\epsilon,\delta\in(0,1)$ be arbitrarily chosen, and let $K$ be a pre-chosen maximum iteration number for running  Algorithm \ref{alg:ac-dq-f}. Let $L(\cdot)$, $D_h$, {$\sigma_1$, $\Lambda(\cdot)$, $\bar\Lambda(\cdot,\cdot)$, $\Lambda_1(\cdot)$, and $\Lambda_2(\cdot)$} be defined in \eqref{def:inexact-Lip}, \eqref{def:diam-h}, {\eqref{def:sigma12}-\eqref{def:C32}}, respectively, {$\bar{B}(\cdot)$, $B_{1}(\cdot)$, $B_{2}(\cdot)$ be defined in \eqref{def:BC31},} $L_f,M_f$ be given in Assumption \ref{asp:basic}(a), and let
\begin{align}
\eta & = \min\bigg\{\frac{1}{4(L_f + L(\epsilon/K))},\bigg(\frac{6}{(M_f^2 + {\Lambda}(\epsilon)^2)(2K+3)(K+2)K}\bigg)^{\frac{1}{2}}D_h\bigg\},\label{def:eta-apsg-c}\\ 
{\bar\eta} & = \min\bigg\{\frac{1}{4(L_f + L(\epsilon/K))},\bigg(\frac{2}{(M_f^2 + {\bar{\Lambda}}(\delta,\epsilon)^2)(K+2)^2K}\bigg)^{\frac{1}{2}}D_h\bigg\},\label{def:eta-apsg-c-1}\\
{\tilde\eta} &\ {= \min\bigg\{\frac{1}{4(L_f + L(\epsilon/K))},\bigg(\frac{2}{\big(M_f^2 + 2\sigma_1^2 + \frac{\Lambda_1(\delta)^2}{2D_h^2K^{1/2}}+\frac{\Lambda_2(\delta)^2}{2D_h^2K}\big)(K+2)^2K}\bigg)^{\frac{1}{2}}D_h\bigg\}}.\label{def:eta-apsg-c-2}
\end{align}
Then the following statements hold.
\begin{enumerate}[{\rm (i)}]
\item Let $\{z^k\}$ be generated by Algorithm \ref{alg:ac-dq-f} with $(\gamma_k,\eta_k)=(2/(k+2), (k+2)\eta/2)$ for all $k\ge0$. Then, $\E[F(z^K) - F^*]\le\epsilon$ for all $K$ satisfying
\begin{align}
K\ge \max\bigg\{\Big(\frac{48D_h^2L_f}{\epsilon}\Big)^\frac{1}{2},\Big(\frac{48D_h^2L(\epsilon)}{\epsilon}\Big)^{\frac{1+\nu}{1+3\nu}}, \frac{(24D_h)^2(M_f+\Lambda(\epsilon))^2}{3\epsilon^2},2\bigg\}. \label{K3}
\end{align}
\item {Let $\{z^k\}$ be generated by Algorithm \ref{alg:ac-dq-f} with $(\gamma_k,\eta_k)=(2/(k+2), (k+2)\bar\eta/2)$ for all $k\ge0$. Then, with probability at least $1-\delta$, $F(z^K)-F^*\le\epsilon$ holds for all $K$ satisfying}
\begin{align}
{
K\ge \max\bigg\{\Big(\frac{64D_h^2L_f}{\epsilon}\Big)^{\frac{1}{2}},\Big(\frac{64D_h^2L(\epsilon)}{\epsilon}\Big)^{\frac{1+\nu}{1+3\nu}}, \frac{2(16D_h)^2(M_f+{\bar\Lambda}(\delta,\epsilon))^2}{\epsilon^2},\Big(\frac{16\bar{B}(\delta)}{\epsilon}\Big)^{\frac{\alpha}{\alpha-1}},2\bigg\}.} \label{K33}
\end{align}
\item Suppose additionally that Assumption \ref{asp:sub-weibull} holds. Let $\{z^k\}$ be generated by Algorithm \ref{alg:ac-dq-f} with $(\gamma_k,\eta_k)=(2/(k+2), (k+2){\tilde\eta}/2)$ for all $k\ge0$. Then, with probability at least $1-\delta$, $F(z^K) - F^*\le\epsilon$ holds for all $K$ satisfying
\begin{align}
&{
K\ge \max\bigg\{\Big(\frac{64D_h^2L_f}{\epsilon}\Big)^{\frac{1}{2}}, \Big(\frac{64D_h^2L(\epsilon)}{\epsilon}\Big)^{\frac{1+\nu}{1+3\nu}},\Big(\frac{16(2D_h(M_f+\sigma_1)+B_{1}(\delta))}{\epsilon}\Big)^2,\Big(\frac{16\Lambda_{1}(\delta)}{\epsilon}\Big)^{\frac{4}{3}},\Big(\frac{16B_2(\delta)}{\epsilon}\Big)^{\max\{\beta,1\}}, }\nonumber\\
&{\qquad\qquad\qquad\qquad\qquad\qquad\qquad\qquad\qquad\qquad\qquad\qquad\qquad\qquad\qquad\qquad\qquad\qquad\qquad \frac{16\Lambda_2(\delta)}{\epsilon},2\bigg\}.}  \label{K4}
\end{align}
\end{enumerate}
\end{theorem}

\begin{remark}
{(i) By virtue of \eqref{def:inexact-Lip}, \eqref{def:sigma1}, \eqref{def:BC31}, and Theorem~\ref{thm:cv-apsg}(i)-(ii), we observe that, under Assumption~\ref{asp:basic}, Algorithm~\ref{alg:ac-dq-f} achieves the first-order oracle complexities in \eqref{cplx-a2-e} and \eqref{cplx-2-p} for finding an $\epsilon$-stochastic optimal solution and an $(\epsilon,\delta)$-stochastic optimal solution of \eqref{ucpb}, respectively. These complexity bounds improve upon those in \eqref{cplx-1-e} and \eqref{cplx-1-p} achieved by SPGM in terms of their dependence on $\epsilon$, $D_h$, $L_f$, $H_f$, and $\nu$. Moreover, \eqref{cplx-a2-e} exhibits universally optimal dependence on $\epsilon$ for smooth, weakly smooth, and nonsmooth convex problems, matching the dependence achieved by deterministic first-order methods \cite{ghadimi2019generalized,nemirovski1983problem,nesterov2015universal} and stochastic first-order methods under heavy-tailed noise \cite{fatkhullin2025can,liu2023stochastic}. 
In addition, by \eqref{def:C32}, \eqref{def:BC31}, and Theorem~\ref{thm:cv-apsg}(iii), we see that, under the additional sub-Weibull noise condition, Algorithm~\ref{alg:ac-dq-f} achieves the first-order oracle complexity in \eqref{subwei-cmplx-aspg}, which depends only polylogarithmically on $1/\delta$. This complexity bound also improves upon the corresponding bound in \eqref{subwei-cmplx-spg} achieved by SPGM in terms of its dependence on $\epsilon$, $D_h$, $L_f$, $H_f$, and $\nu$.

(ii) A key technical contribution to establishing the universal near-optimal complexity of accelerated SPGM stated in Theorem~\ref{thm:cv-apsg} is the construction of the novel inexact proxies $\Lambda(\varepsilon)^2$ and $\bar{\Lambda}(\delta,\varepsilon)^2$ for the noise variance, as defined in \eqref{def:sigma1}. This construction is inspired by the inexact-oracle framework developed for universal first-order methods for weakly smooth optimization in \cite{nesterov2015universal}. In addition, the results in Theorem~\ref{thm:cv-apsg}(ii)-(iii) are established using the new concentration inequalities in Lemmas~\ref{lem:hts} and~\ref{lem:mds}, respectively.

(iii) Under the sub-Weibull assumption, accelerated SPGM achieves high-probability complexity bounds with a polylogarithmic dependence on $1/\delta$, similar to that achieved by clipping-based methods under only heavy-tailed noise; see, e.g., \cite{gorbunov2020stochastic,liu2023stochastic,nguyen2023improved,sadiev2023high,zhang2022parameter}. However, to the best of our knowledge, existing results on clipping-based methods do not directly provide the unified accelerated complexity guarantees established here for composite problems with smooth, weakly smooth, and nonsmooth components and a general proximal term with bounded domain. Our bounds separately capture the effects of these components on the convergence of accelerated SPGM without gradient clipping. 
}
\end{remark}

\section{Numerical experiments} \label{sec:experiment} 



{In this section, we present preliminary numerical experiments to evaluate the performance of Algorithms~\ref{alg:dq-f} and \ref{alg:ac-dq-f}, with the latter referred to as SPGM-A. We also compare these methods with a clipped variant of SPGM, denoted by SPGM-C. All algorithms are implemented in \textsc{Matlab}, and all experiments are conducted on a laptop equipped with an Intel Core i9-14900HX processor (2.20~GHz) and 32~GB of RAM.

In our experiments, we simulate stochastic gradients by artificially adding heavy-tailed random perturbations to exact gradient evaluations. The nature of the corresponding noise is essentially the same as that arising in practical applications involving heavy-tailed noise. For example, in privacy-preserving optimization, heavy-tailed noise may be deliberately injected into the actual gradients \cite{simsekli2024differential}. As another example, in communication systems, the actual gradient may be contaminated by heavy-tailed communication noise during transmission \cite{gulgun2023massive,zhang2021measurements}.
}

\subsection{$\ell_1$-regularized $\ell_2$-$\ell_p$ regression with box constraints}\label{sec:prox-ne}

In this subsection, we consider the $\ell_1$-regularized $\ell_2$-$\ell_p$ regression with box constraints:
\begin{align}\label{pb:lp-l1r}
\min_{l\le x\le u} \frac{1}{2}\|Ax-b\|^2 + \frac{1}{p}\|Ax-b\|_p^p + \lambda \|x\|_1,   
\end{align}
where $A \in \mathbb{R}^{n \times n}$, $b \in \mathbb{R}^n$, $u=-l=100\cdot\mathbf{1}$ with $\mathbf{1}$ being the all-ones vector, $p = 1.5$, and $\lambda = 1$. {Problem \eqref{pb:lp-l1r} can be viewed as an instance of \eqref{ucpb} with $f(x)=\frac{1}{2}\|Ax-b\|^2 + \frac{1}{p}\|Ax-b\|_p^p$ and $h(x)=\lambda \|x\|_1+\iota_{[l,u]}(x)$.} We simulate noisy gradient evaluations by setting the stochastic gradient estimator as $G(x;\xi)=\nabla f(x) + \rho\xi$, where $\rho>0$ is a deterministic scalar, and $\xi\in\R^n$ has independently distributed coordinates, each following a heavy-tailed distribution with density function $p(t) = \omega/(2(1 + |t|)^{1+\omega})$. One can verify that such $G(\cdot;\xi)$ satisfies Assumption \ref{asp:basic}(c) for every $\alpha\in(1,\omega)$, and that the $\alpha$th central moment of $G(\cdot;\xi)$ is unbounded for all $\alpha\ge\omega$.

For each triple $(n,\rho,\omega)$, we randomly generate 10 instances of problem~\eqref{pb:lp-l1r}. In particular, we first randomly generate $A$ with all its elements sampled from the standard normal distribution. We then randomly generate $\Bar{x}_{\text{ref}}$, with all its components sampled from the standard normal distribution, and construct a sparse ground-truth solution $x_{\text{true}}$ by randomly setting half of its components to zero. Finally, we set $b=Ax_{\text{true}}$.

We apply SPGM, SPGM-A, and SPGM-C to problem \eqref{pb:lp-l1r} to find an approximate solution $x^k$ such that its relative objective value gap $(F(x^k) - F^*)/(F(x^0) - F^*)$ is less than $10^{-4}$, where $F^*$ is estimated by CVX \cite{grant2008cvx}. All methods are initialized at the zero vector. 
{
The other algorithmic parameters are tuned according to the following strategy:
\begin{itemize}
    \item For SPGM, we choose step sizes $\eta_k=\eta/(k+1)^{1/2}$ for all $k\ge0$.
    \item For SPGM-A, we choose step sizes $\eta_k=(k+2)\eta/[2(k+1)^{1/2}]$ and weighting parameters $\gamma_k=2/(k+2)$ for all $k\ge0$.
    \item For SPGM-C, we choose step sizes $\eta_k=\eta/(k+1)^{1/2}$ for all $k\ge0$ and clipping threshold $\tau$.
\end{itemize}
For each method and problem instance, the stepsize-related  parameter \(\eta\) is selected by grid search over the common candidate set $\{0.0025, 0.005, 0.01, 0.02, 0.025\}$. For SPGM-C, the clipping threshold $\tau$ is also selected by grid search over $\{250, 500, 1000\}$ for each problem instance.
}


The computational results of SPGM, SPGM-A, and SPGM-C for solving \eqref{pb:lp-l1r} are presented in Table~\ref{table:pb1}. Specifically, the first three columns list the values of $n$, $\rho$, and $\omega$, respectively, while the remaining columns report the average CPU time and the average number of iterations for each triple $(n,\rho,\omega)$. We observe that, except for the case $(\rho,\omega) = (100,1.2)$, SPGM-A substantially outperforms both SPGM and SPGM-C. When $(\rho,\omega) = (100,1.2)$, the performance gap between SPGM-A and SPGM-C is considerably smaller, although both methods still outperform SPGM. These observations suggest that when the noise level is not too high, the accelerated SPGM can achieve significantly faster convergence than the vanilla SPGM and the clipped SPGM, which is consistent with our theoretical findings.

{

\begin{table}[!htbp]
\centering
{\small
\begin{tabular}{ccc||rrr||rrr}
\hline
& & &\multicolumn{3}{c||}{CPU time (seconds)} &\multicolumn{3}{c}{Iterations} \\
$n$ & $\rho$ & $\omega$ & SPGM & SPGM-A & SPGM-C & SPGM & SPGM-A & SPGM-C \\ \hline
500 & 1 & 1.8 & 3.15 & 1.50 & 3.16 & 2602.2 & 1284.0 & 2606.8 \\ 
500 & 1 & 1.5 & 3.54 & 1.65 & 3.68 & 2734.8 & 1280.3 & 2727.6 \\ 
500 & 1 & 1.2 & 3.53 & 1.73 & 3.76 & 2810.5 & 1311.6 & 2825.6 \\ 
500 & 100 & 1.8 & 3.28 & 1.57 & 3.24 & 2528.1 & 1233.8 & 2520.5 \\ 
500 & 100 & 1.5 & 4.38 & 2.22 & 3.90 & 3226.5 & 1711.5 & 2689.3 \\ 
500 & 100 & 1.2 & 4.65 & 3.16 & 3.22 & 4158.1 & 2740.8 & 2872.0 \\ 
1000 & 1 & 1.8 & 18.82 & 8.11 & 19.40 & 5006.9 & 2130.6 & 5040.6 \\ 
1000 & 1 & 1.5 & 20.97 & 8.99 & 21.69 & 4952.5 & 2149.6 & 5059.7 \\ 
1000 & 1 & 1.2 & 21.14 & 8.77 & 22.52 & 5085.2 & 2175.4 & 5322.0 \\ 
1000 & 100 & 1.8 & 20.90 & 8.95 & 21.68 & 5154.1 & 2174.9 & 5355.7 \\ 
1000 & 100 & 1.5 & 24.09 & 9.32 & 22.72 & 5695.5 & 2167.0 & 5048.7\\ 
1000 & 100 & 1.2 & 24.85 & 19.08 & 22.41 & 6779.0 & 5165.0 & 6106.3 \\
\hline 
\end{tabular}
}
\caption{Numerical results for problem~\eqref{pb:lp-l1r}.}
\label{table:pb1}
\end{table}

In the above experiment, the stepsize-related parameter $\eta$ for SPGM, SPGM-A, and SPGM-C is selected by grid search. In practice, however, one may use an arbitrary value of $\eta$ rather than the best value identified by grid search. We therefore next investigate the sensitivity of these methods to the choice of $\eta$. To this end, we apply each method to a representative problem instance with $(n,\rho,\omega)=(1000,1,1.8)$ using different values of $\eta$ and record the number of iterations and CPU time required to achieve a relative objective value gap below $10^{-2}$. The results are reported in Figure~\ref{fig:prox-SPGM}. We observe that SPGM, SPGM-A, and SPGM-C achieve their best performance at $\eta=0.005$, $\eta=0.01$, and $(\tau,\eta)=(1000,0.005)$, respectively. In addition, SPGM-C appears to be highly sensitive to the choice of $\eta$, as a slight increase in $\eta$ leads to a significant deterioration in performance. In contrast, SPGM and SPGM-A exhibit more stable performance across different values of $\eta$.

\begin{figure}[!htbp]
\centering
\includegraphics[width=.45\linewidth]{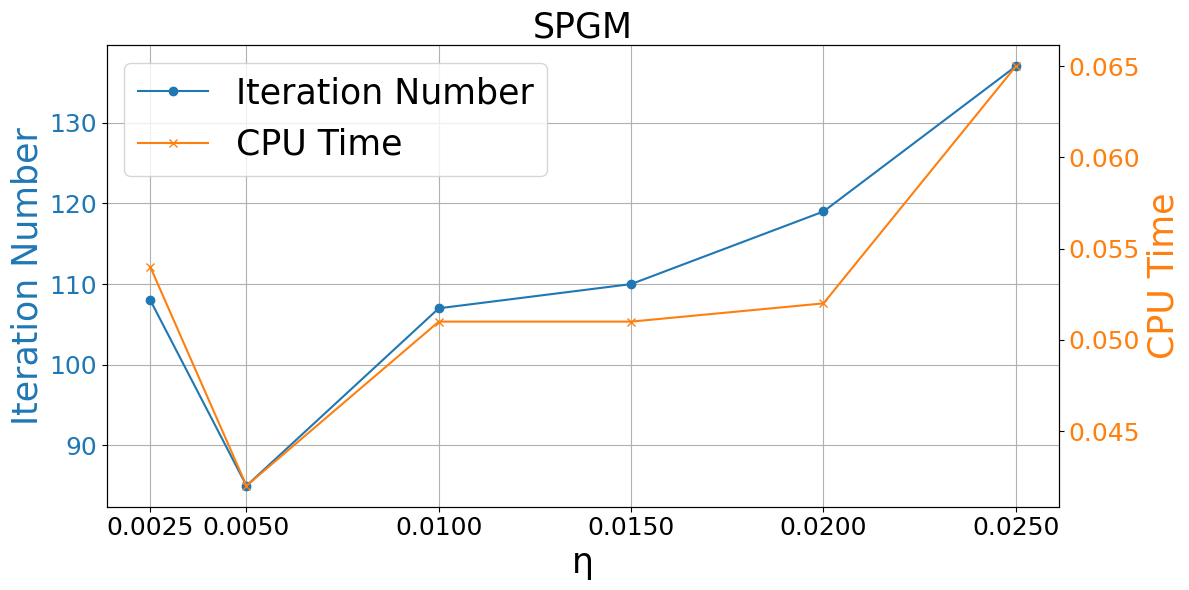}
\includegraphics[width=.45\linewidth]{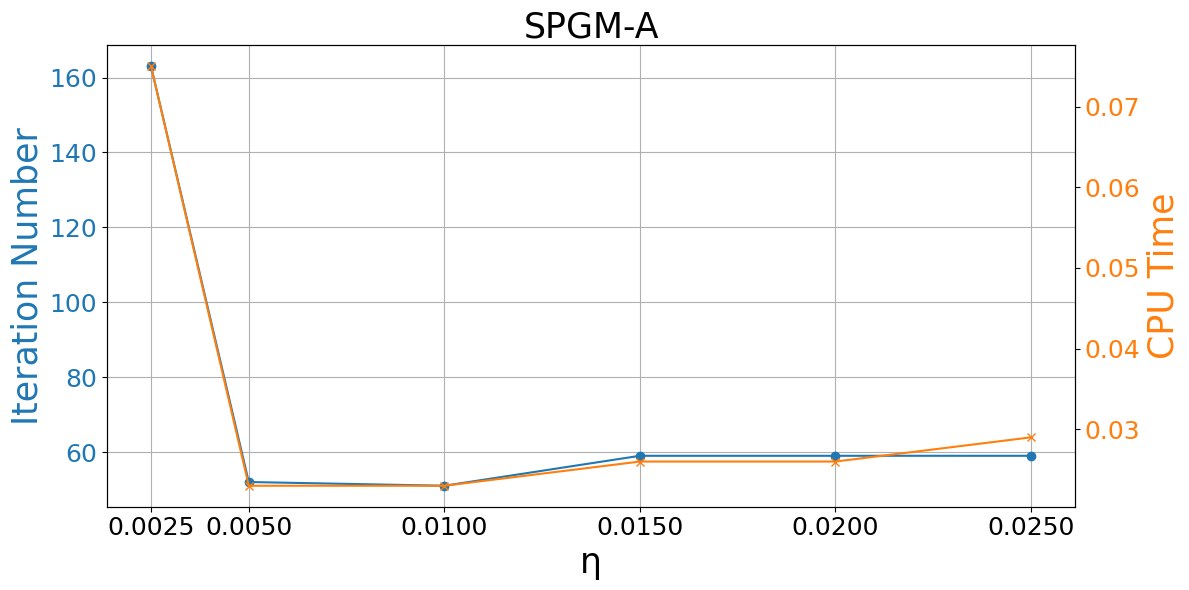}\\
\includegraphics[width=.45\linewidth]{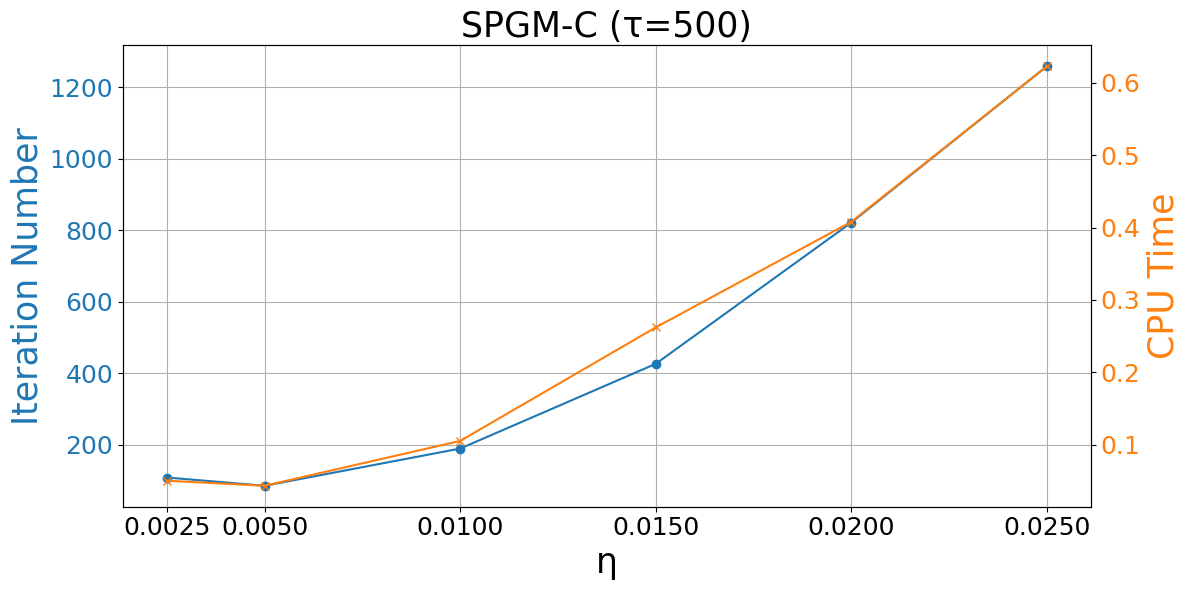}
\includegraphics[width=.45\linewidth]{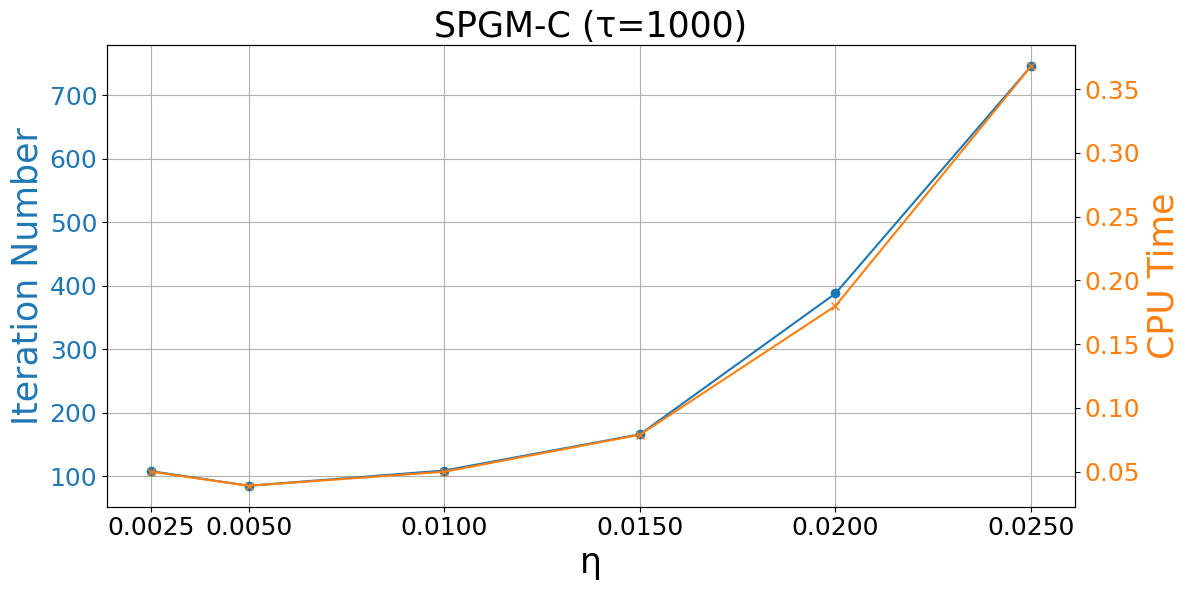}
\caption{{Sensitivity of SPGM, SPGM-A, and SPGM-C to the choice of $\eta$.}}
\label{fig:prox-SPGM}
\end{figure}

}

\subsection{$\ell_2$-$\ell_p$-$\ell_1$ regression with $\ell_2$-ball constraint}
In this subsection, we consider the $\ell_2$-$\ell_p$-$\ell_1$ regression with $\ell_2$-ball constraint:
\begin{align}\label{pb:lp-l1-l2}
\min_{\|x\|\le u} \frac{1}{2}\|Ax-b\|^2 +\frac{1}{p}\|Ax-b\|_p^p + \lambda \|Ax-b\|_1,   
\end{align}
where $A\in\R^{n\times n}$, $b\in\R^n$, $u=100$, $p=1.5$, and $\lambda=0.1$. {Problem \eqref{pb:lp-l1-l2} can be viewed as an instance of problem~\eqref{ucpb} with $f(x)=\frac{1}{2}\|Ax-b\|^2 +\frac{1}{p}\|Ax-b\|_p^p + \lambda \|Ax-b\|_1$ and $h(x)=\iota_{\{\|x\|\le u\}}(x)$.} We simulate the noisy gradient evaluations by setting the stochastic gradient estimator as $G(x;\xi)= g(x) + \rho\xi$, where $g(x)\in \partial f(x)$ and $\rho>0$ is a deterministic scalar, and $\xi\in\R^n$ has independently distributed coordinates, each following a heavy-tailed distribution with density function $p(t) = \omega/(2(1 + |t|)^{1+\omega})$. One can verify that such $G(\cdot;\xi)$ satisfies Assumption \ref{asp:basic}(c) for every $\alpha\in(1,\omega)$, and that the $\alpha$th central moment of $G(\cdot;\xi)$ is unbounded for all $\alpha\ge\omega$.

For each triple $(n,\rho,\omega)$, we randomly generate 10 instances of problem~\eqref{pb:lp-l1-l2}. In particular, we first randomly generate $A$ with all its elements sampled from the standard normal distribution. We then randomly generate $x_{\text{true}}$, with all its components sampled from the standard normal distribution, and set $b=Ax_{\text{true}}$.

We apply SPGM, SPGM-A, and SPGM-C to problem \eqref{pb:lp-l1-l2} to find an approximate solution $x^k$ such that its relative objective value gap $F(x^k)/F(x^0)$ is less than $10^{-4}$ (note that $F^*=0$ due to the data generation setup). All methods are initialized at the zero vector. { The other algorithmic parameters are tuned according to the strategy described in Subsection~\ref{sec:prox-ne}.

\begin{table}[tbhp]
\centering
\begin{tabular}{ccc||rrr||rrr}
\hline
& & &\multicolumn{3}{c||}{CPU time (seconds)} &\multicolumn{3}{c}{Iterations} \\
$n$ & $\rho$ & $\omega$ & SPGM & SPGM-A & SPGM-C & SPGM & SPGM-A & SPGM-C \\ \hline
500 & 1 & 1.8 & 18.00 & 9.65 & 26.80 & 9772.6 & 3640.0 & 9374.6 \\ 
500 & 1 & 1.5 & 12.10 & 4.11 & 11.76 & 9126.5 & 3167.2 & 8659.1 \\ 
500 & 1 & 1.2 & 12.02 & 3.95 & 11.85 & 8838.4 & 2936.4 & 8319.8 \\ 
500 & 100 & 1.8 & 13.30 & 5.06 & 13.42 & 9671.7 & 3594.2 & 9411.8 \\ 
500 & 100 & 1.5 & 16.32 & 5.19 & 13.11 & 12190.4 & 3828.9 & 9640.9 \\ 
500 & 100 & 1.2 & 21.36 & 12.54 & 13.00 & 15843.0 & 9051.7 & 9730.7 \\ 
1000 & 1 & 1.8 & 48.61 & 18.62 & 59.23 & 8524.6 & 2769.9 & 8512.5 \\ 
1000 & 1 & 1.5 & 39.42 &  13.03 & 39.54 & 8305.2 & 2753.9 & 8300.3 \\ 
1000 & 1 & 1.2 & 50.64 & 12.66 & 40.22 & 10454.3 & 2700.5 & 8386.8 \\ 
1000 & 100 & 1.8 & 40.62 & 13.83 & 43.13 & 8126.3 & 2638.7 & 8119.6 \\ 
1000 & 100 & 1.5 & 82.54 & 24.76 & 77.66 & 10729.2 & 2871.3 & 9076.5 \\ 
1000 & 100 & 1.2 & 124.10 & 52.07 & 80.57 & 17992.7 & 8854.2 & 9496.5 \\
\hline 
\end{tabular}
\caption{Numerical results for problem~\eqref{pb:lp-l1-l2}.}
\label{table:pb2}
\end{table}

The computational results of SPGM, SPGM-A, and SPGM-C for solving \eqref{pb:lp-l1-l2} are presented in Table~\ref{table:pb2}. The first three columns list the values of $n$, $\rho$, and $\omega$, respectively, while the remaining columns report the average CPU time and number of iterations for each triple $(n,\rho,\omega)$. We observe that, except for the case $(\rho,\omega)=(100,1.2)$, SPGM-A significantly outperforms both SPGM and SPGM-C. When $(\rho,\omega)=(100,1.2)$, the performance gap between SPGM-A and SPGM-C narrows, although both methods still outperform SPGM. These results suggest that, when the noise level is moderate, the accelerated SPGM can converge substantially faster than both the vanilla and clipped SPGMs, consistent with our theoretical results.

Motivated by the same considerations as in Subsection~\ref{sec:prox-ne}, we next investigate the sensitivity of SPGM, SPGM-A, and SPGM-C to the choice of~$\eta$. Specifically, we apply each method to a representative problem instance with $(n,\rho,\omega)=(1000,1,1.8)$ using different values of $\eta$ and record the number of iterations and CPU time required to achieve a relative objective gap below $10^{-2}$. The results are reported in Figure~\ref{fig:ball-SPGM}. As shown in the figure, SPGM, SPGM-A, and SPGM-C achieve their best performance at $\eta=0.005$, with the best clipping threshold for SPGM-C being $\tau=1000$. Furthermore, SPGM-C appears to be highly sensitive to the choice of~$\eta$, as a modest increase in $\eta$ leads to a substantial deterioration in performance. In contrast, SPGM and SPGM-A exhibit greater robustness across different values of~$\eta$.}

\begin{figure}[!htbp]
\centering
\includegraphics[width=.45\linewidth]{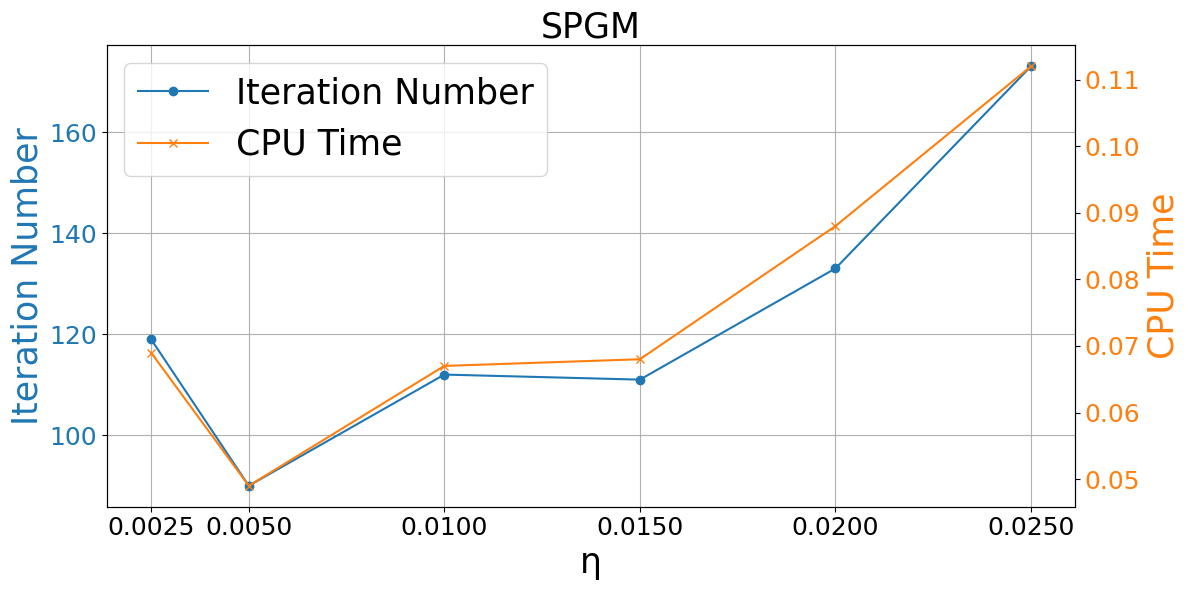}
\includegraphics[width=.45\linewidth]{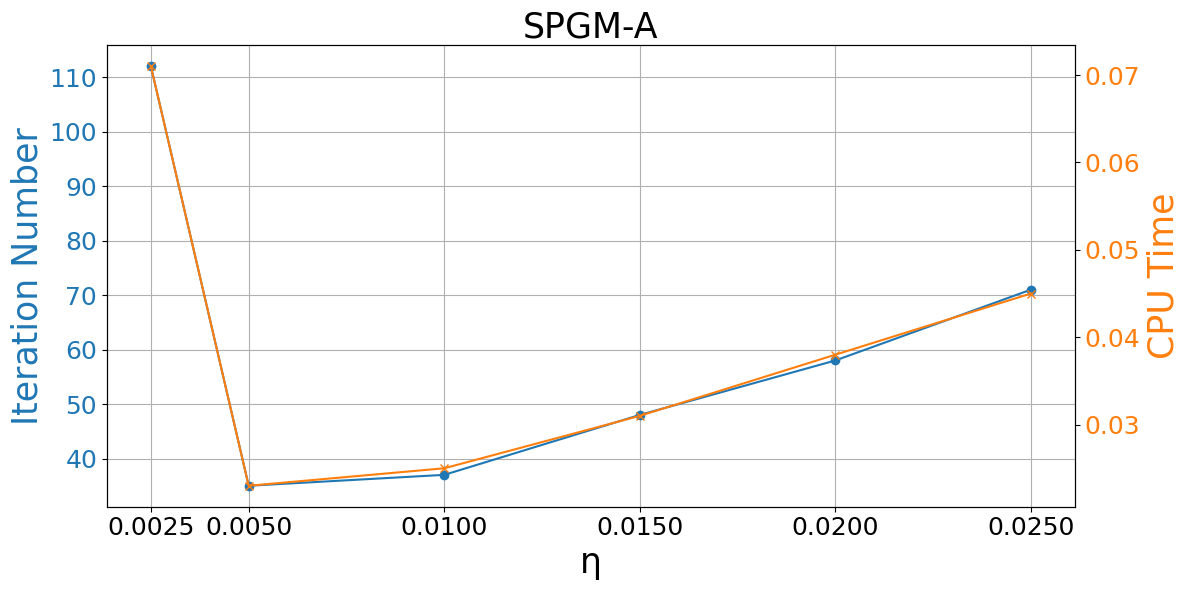}\\
\includegraphics[width=.45\linewidth]{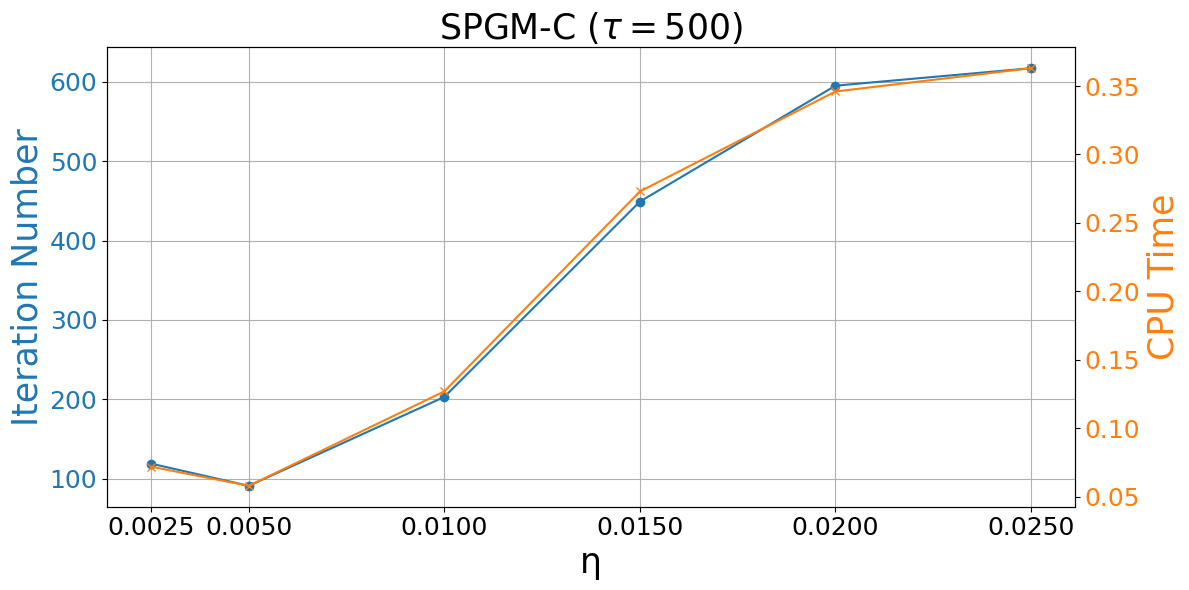}
\includegraphics[width=.45\linewidth]{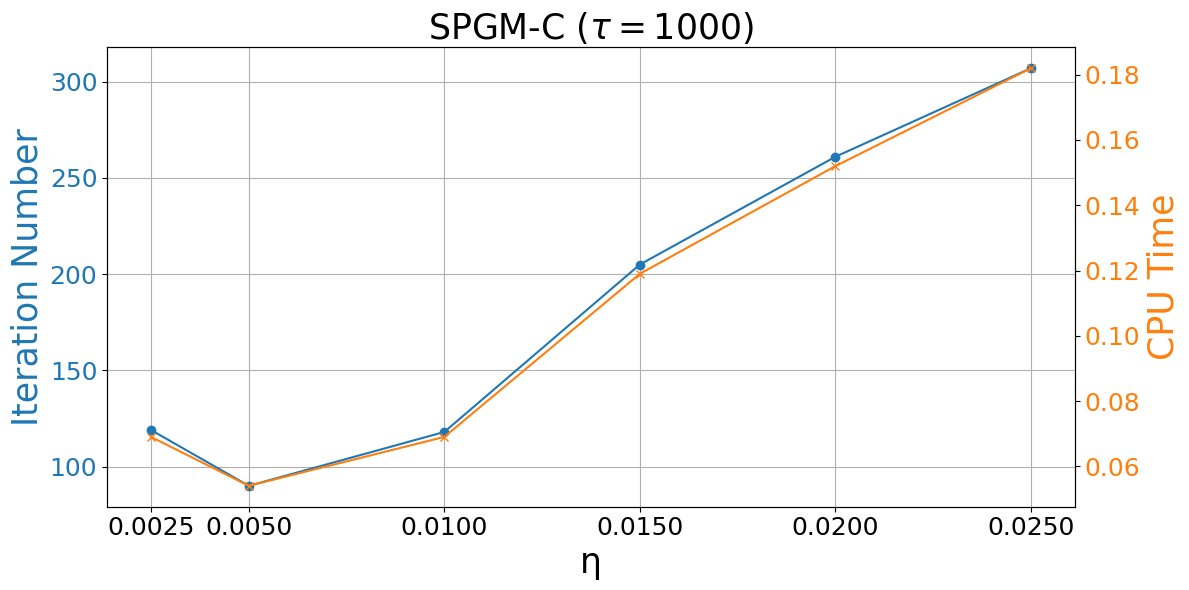}
\caption{{Sensitivity of SPGM, SPGM-A, and SPGM-C to the choice of $\eta$.}}
\label{fig:ball-SPGM}
\end{figure}

\section{Proof of the main results} \label{sec:proof}

In this section, we present the proofs of the main results stated in Sections~\ref{sec:psg} and~\ref{sec:apsg}, namely, Theorems~\ref{thm:cv-psg} and~\ref{thm:cv-apsg}. Throughout this section, let $x^*$ denote an arbitrary but fixed optimal solution to \eqref{ucpb}.

\subsection{{Proof of several concentration inequalities}}

{

In this subsection, we develop concentration inequalities for martingale difference sequences with bounded $\alpha$th moments and sub-Weibull tails. These results will serve as key tools in our subsequent complexity analysis for establishing high-probability guarantees.\footnote{{For simplicity of exposition, we present only the fixed-$K$ versions of the concentration inequalities, as these suffice for our algorithmic analysis. Stronger maximal versions, which control all partial sums up to $K$ simultaneously, can be obtained by replacing Markov's inequality with Doob's maximal inequality in the proofs.}}

The following lemma presents a new concentration inequality for martingale difference sequences with bounded moments of order $\alpha$. It will be used to establish high-probability complexity bounds. 

\begin{lemma}\label{lem:hts}
Let $\{\xi_k\}$ be a sequence of i.i.d. random variables, let $\varsigma>0$ and $\alpha\in(1,2]$, and let $\{\phi_k(\cdot)\}$ be a sequence of deterministic functions. Define $\phi^k = \phi_k(\xi_{[k]})$ with $\xi_{[k]}=\{\xi_t\}_{t=0}^k$ for all $k=0,1,\ldots$, and let $\E_{\xi_0}[\cdot\,|\, \xi_{[-1]}]:=\E_{\xi_0}[\cdot]$. Suppose that, for every $k\ge0$, $\E_{\xi_k}[\phi^k\,|\,\xi_{[k-1]}]=0$ and $\E_{\xi_k}[|\phi^k|^\alpha\,|\,\xi_{[k-1]}] \le \varsigma^\alpha$ hold. Then, for every positive integer $K$ and every $t>0$,
\begin{align}\label{ineq:hts}
\mathbb{P}\bigg(\bigg|\sum_{k=0}^{K-1}\phi^k\bigg|>2^{\frac{2-\alpha}{\alpha}}\varsigma K^{\frac{1}{\alpha}}e^{\frac{t}{\alpha}}\bigg) \le e^{-t}.
\end{align}
\end{lemma}

\begin{proof}
Fix a positive integer $K$ and $t>0$, and define
\[
S_0:=0,
\qquad
S_k:=\sum_{i=0}^{k-1}\phi^i,
\quad k=1,\ldots,K.
\]
We first bound $\E[|S_K|^\alpha]$ and then apply Markov's inequality.

Let $g(u):=|u|^\alpha$. Its derivative is $(\alpha-1)$-H\"older continuous: $|g'(u)-g'(v)| \le \alpha 2^{2-\alpha}|u-v|^{\alpha-1}$ for all $u,v\in\R$; see, e.g., \cite[Theorem 6.3]{rodomanov2020smoothness}. The fundamental theorem of calculus therefore implies
\begin{align*}
    |u+v|^\alpha
\le |u|^\alpha+g'(u)v+2^{2-\alpha}|v|^\alpha\qquad \forall u,v\in\mathbb R.
\end{align*}
Recall that $S_k$ and $\phi^k$ have finite $\alpha$th moments. Then, $g'(S_k)\phi^k$ is integrable by H\"older's inequality. For $k=0,\ldots,K-1$, applying the above inequality with $(u,v)=(S_k,\phi^k)$ and taking conditional expectations given $\xi_{[k-1]}$ yields
\begin{align*}
\E_{\xi_{k}}[|S_{k+1}|^\alpha\mid\xi_{[k-1]}]
\le |S_k|^\alpha
+g'(S_k)\E_{\xi_{k}}[\phi^k\mid\xi_{[k-1]}]+2^{2-\alpha}\E_{\xi_{k}}[|\phi^k|^\alpha\mid\xi_{[k-1]}] \le |S_k|^\alpha+2^{2-\alpha}\varsigma^\alpha.
\end{align*}
Taking expectations and iterating from $S_0=0$ gives $\E[|S_K|^\alpha]\le 2^{2-\alpha}K\varsigma^\alpha$. Hence, Markov's inequality yields
\[
\mathbb P\left(
|S_K|>
2^{\frac{2-\alpha}{\alpha}}
\varsigma K^{\frac1\alpha}e^{\frac t\alpha}
\right)
\le
\frac{\E[|S_K|^\alpha]}
{2^{2-\alpha}K\varsigma^\alpha e^t}
\le e^{-t},
\]
which proves \eqref{ineq:hts}.
\end{proof}

}

The following lemma provides a concentration inequality for a martingale difference sequence of sub-Weibull random variables. It will be used to establish high-probability complexity bounds. {For $\beta\in(0,1)$, this inequality improves upon the related result in \cite[Theorem 11]{madden2024high} by removing the additional logarithmic dependence on $K$.}

\begin{lemma}\label{lem:mds}
Let $\{\xi_k\}$ be a sequence of i.i.d. random variables, $\varsigma>0$ and ${\beta\in(0,2]}$ be given, and $\{\phi_k(\cdot)\}$ be a sequence of deterministic functions. Define $\phi^k = \phi_k(\xi_{[k]})$ with $\xi_{[k]}=\{\xi_t\}_{t=0}^k$ for all $k=0,1,\ldots$, and let $\E_{\xi_0}[\cdot\,|\, \xi_{[-1]}]:=\E_{\xi_0}[\cdot]$. Suppose that, for every $k\ge0$, $\E_{\xi_k}[\phi^k\,|\,\xi_{[k-1]}]=0$ and $\E_{\xi_k}[\exp\{|\phi^k/\varsigma|^{\beta}\}\,|\,\xi_{[k-1]}] \le e$ hold. Then, for every positive integer $K$ and every $t>0$,
\begin{align}\label{ineq:mds}
{\mathbb{P}\bigg(\bigg|\sum_{k=0}^{K-1}\phi^k\bigg|>4e^{\frac{3}{2}}\Big(\frac{4}{\beta}\Big)^{\frac{1}{\beta}}\varsigma \big(\sqrt{K(t+1)} + K^{\max\{1-\frac{1}{\beta},0\}}(t+1)^{\frac{1}{\beta}}\big)\bigg) \le e^{-t}.}
\end{align}
\end{lemma}
{
\begin{proof}
Fix a positive integer $K$ and $t>0$, and set
$u^k:=\phi^k/\varsigma$ for all $k\ge0$. By conditional Markov's inequality, one has that for every $k\ge0$,
\begin{align}
\mathbb P(|u^k|>v\mid\xi_{[k-1]})=
\mathbb P\left(
e^{|u^k|^\beta}>e^{v^\beta}
\,\middle|\,\xi_{[k-1]}
\right)\le\min\left\{1,\,e^{-v^\beta}
\E_{\xi_k}[e^{|u^k|^\beta}\mid\xi_{[k-1]}]
\right\}\le\min\{1,e^{1-v^\beta}\}\qquad\forall v>0.
\label{ineq:prob-tail-uk}
\end{align}
Let $\zeta$ be an exponential random variable with mean $1$, independent of $\{\xi_k\}$, and define $\omega:=(1+\zeta)^{1/\beta}$. Then,
\begin{equation}\label{ineq:prob-tail-w}
\mathbb P(\omega>v)
= \mathbb{P}(\zeta> v^\beta - 1) = \min\{1,e^{1-v^\beta}\}
\qquad \forall v\ge0.
\end{equation}
Let $\{\zeta_k\}$ be independent copies of $\zeta$, and let $\{s_k\}$ be independent Rademacher random variables. Assume that $\{\zeta_k\}$, $\{s_k\}$, and $\{\xi_k\}$ are mutually independent, and define $\tilde\omega_k:=2s_k(1+\zeta_k)^{1/\beta}$ for all $k\ge0$. The random variables $\{\tilde\omega_k\}$ are independent and symmetric about zero, with finite moments of every order. In view of \eqref{ineq:prob-tail-uk}-\eqref{ineq:prob-tail-w}, we apply Lemma~\ref{lem:cvx-comparison} with $h(x)=x^{2m}$
to obtain, for every positive integer $m$,
\begin{equation}\label{ineq:h-expect-upbd}
\E\bigg[\bigg|\sum_{k=0}^{K-1}u^k\bigg|^{2m}\bigg] \le
\E\bigg[\bigg|\sum_{k=0}^{K-1}\tilde\omega_k\bigg|^{2m}\bigg].
\end{equation}

We next bound the moments of $\tilde\omega_k$. For $p>0$ and $y\ge0$, the elementary inequality
$y^p\le(\frac{2p}{e})^p e^{y/2}$ follows by maximizing $y^p e^{-y/2}$. Since $\E[e^{(1+\zeta_k)/2}]=2\sqrt e$, we obtain that for every integer $j\ge1$,
\begin{align*}
\E[\tilde\omega_k^{2j}]
=
4^j\E[(1+\zeta_k)^{\frac{2j}{\beta}}] \le
2\sqrt e\,4^j
\Big(\frac{4j}{\beta e}\Big)^{\frac{2j}{\beta}} \le
\bigg[16\Big(\frac4\beta\Big)^{\frac{2}{\beta}}\bigg]^j
(j!)^{\frac{2}{\beta}},
\end{align*}
where the last inequality uses $2\sqrt e\le4^j$ and $(j/e)^j\le j!$. Moreover, notice that for every integer $j\ge1$,
\[
\frac{(2j)!}{(j!)^2} = \prod_{i=1}^j\Big(1+\frac ji\Big)
\ge2^j.
\]
Set $r:=2/\beta-2\ge-1$. Combining this inequality with the preceding moment bound yields, for every integer $j\ge1$,
\begin{equation}\label{ineq:exp-single-omega}
\frac{\E[\tilde\omega_k^{2j}]}{(2j)!}
\le \bigg[16\Big(\frac4\beta\Big)^{\frac{2}{\beta}}\bigg]^j\frac{(j!)^{\frac{2}{\beta}}}{(2j)!} \le \bigg[8\Big(\frac4\beta\Big)^{\frac{2}{\beta}}\bigg]^j(j!)^r.
\end{equation}
This bound also holds for $j=0$, since both sides equal $1$ under the convention $0!=1$. By independence, symmetry, and the multinomial theorem, one has that for every positive integer $m$,
\begin{align}
\E\bigg[\bigg|\sum_{k=0}^{K-1}\tilde\omega_k\bigg|^{2m}\bigg]&=
(2m)!
\sum_{\substack{j_0+\cdots+j_{K-1}=m\\
j_0,\ldots,j_{K-1}\in\mathbb Z_+}}
\prod_{k=0}^{K-1}
\frac{\E[\tilde\omega_k^{2j_k}]}{(2j_k)!} \overset{\eqref{ineq:exp-single-omega}}{\le}
(2m)!
\bigg[8\Big(\frac4\beta\Big)^{\frac{2}{\beta}}\bigg]^m
\sum_{\substack{j_0+\cdots+j_{K-1}=m\\
j_0,\ldots,j_{K-1}\in\mathbb Z_+}}
\prod_{k=0}^{K-1}(j_k!)^r\nonumber\\
& \overset{\eqref{ineq:upper-bd-combin}}{\le} (2m)!\bigg[ 8e\Big(\frac4\beta\Big)^{\frac{2}{\beta}}\bigg(\frac Km+K^{\max\{-r,0\}}m^r\bigg)\bigg]^m \le \bigg[ 32e\Big(\frac4\beta\Big)^{\frac{2}{\beta}}\bigg(
Km+K^{\max\{2-\frac{2}{\beta},0\}}m^{\frac{2}{\beta}}\bigg)\bigg]^m,\label{ineq:upbd-expect-sum-omegak}
\end{align}
where the second inequality follows from Lemma~\ref{lem:frac-bound} and the last inequality is due to $(2m)!\le(2m)^{2m}$ and $r=2/\beta-2$. Finally, define $a:=\max\{1-1/\beta,0\}$. By \eqref{ineq:h-expect-upbd}, \eqref{ineq:upbd-expect-sum-omegak}, and Markov's inequality, one has that for every positive integer $m$,
\begin{align}
\mathbb P\bigg(\bigg|\sum_{k=0}^{K-1}u^k\bigg|>4e^{\frac{3}{2}}\Big(\frac{4}{\beta}\Big)^{\frac{1}{\beta}} \big(\sqrt{Km} + K^am^{\frac{1}{\beta}}\big)\bigg)
&\le \frac{\E\big[|\sum_{k=0}^{K-1}u^k|^{2m}\big]}{\big[4e^{\frac{3}{2}}(4/\beta)^{\frac{1}{\beta}} (\sqrt{Km} + K^am^{\frac{1}{\beta}})\big]^{2m}} \overset{\eqref{ineq:h-expect-upbd}\eqref{ineq:upbd-expect-sum-omegak}}{\le}
\bigg[
\frac{2}{e^2}\cdot
\frac{Km+K^{2a}m^{\frac{2}{\beta}}}{(\sqrt{Km}+K^a m^{\frac{1}{\beta}})^2} \bigg]^m\nonumber\\
&\le
\Big(\frac{2}{e^2}\Big)^m
\le e^{-m}.\label{ineq:concentration-sb-m}
\end{align}
Taking $m=\lceil t\rceil$, we have $t\le m\le t+1$. Thus, 
\begin{align*}
 4e^{\frac{3}{2}}\Big(\frac{4}{\beta}\Big)^{\frac{1}{\beta}} \big(\sqrt{Km} + K^am^{\frac{1}{\beta}}\big)\le
4e^{\frac{3}{2}}\Big(\frac4\beta\Big)^{\frac1\beta}
\big(
\sqrt{K(t+1)}+K^a(t+1)^{\frac1\beta}
\big),
\qquad
e^{-m}\le e^{-t}.
\end{align*}
Combining these bounds with $\phi^k=\varsigma u^k$ and \eqref{ineq:concentration-sb-m} proves \eqref{ineq:mds}.
\end{proof}
}



    

\subsection{Proof of the main results in Section~\ref{sec:psg}}\label{subsec:pf-psg}

In this subsection, we first establish a lemma and then use it to prove Theorem~\ref{thm:cv-psg}.

\begin{lemma}\label{lem:desc-psg}
Suppose that Assumption~\ref{asp:basic} holds. Let $\epsilon\in(0,1)$ be arbitrarily chosen, $L_f$, $M_f$, $\alpha$, and $\sigma$ be given in Assumption~\ref{asp:basic}, and $L(\cdot)$ be defined in \eqref{def:inexact-Lip}. Let $\{x^k\}$ be the sequence generated by Algorithm~\ref{alg:dq-f} with step sizes $\{\eta_k\}$ satisfying $\eta_k\in\big(0, \frac{1}{4(L_f+L(\epsilon))}\big]$ for all $k\ge0$. Then, it holds that for all $k\ge0$, 
\begin{align}
\eta_k(F(x^{k+1}) - F^*)\le &\ (\|x^k-x^*\|^2 - \|x^{k+1}-x^*\|^2)/2 + M_f^2\eta_k^2 + \eta_k\Delta_k \nonumber\\
& + \frac{(8(\alpha-1))^{\alpha-1}D_h^{2-\alpha}\eta_k^\alpha\delta_k^\alpha}{\alpha^\alpha} + \frac{\epsilon\eta_k}{8},\label{ineq:desc-rec} 
\end{align}
where
{
\begin{align}\label{def:zetak-deltak-1}
\Delta_k=(\E_{\xi}[G(x^k;\xi)\,|\,\xi_{[k-1]}]-G(x^k;\xi_k))^T(x^k-x^*),\quad \delta_k=
\|\E_{\xi}[G(x^k;\xi)\,|\,\xi_{[k-1]}]-G(x^k;\xi_k)\|\qquad\forall k\ge0,    
\end{align}
 where $\xi,\xi_0,\xi_1,\ldots$ are i.i.d.\ random variables, $\xi_{[-1]}=\emptyset$,  and  $\xi_{[k]}=\{\xi_t\}_{t=0}^k$ for all $k\ge0$. }
\end{lemma}

\begin{proof}
Fix any $k\ge0$. By the optimality condition of \eqref{prox-step}, there exists $h^\prime(x^{k+1})\in\partial h(x^{k+1})$ such that 
\begin{align*}
G(x^k;\xi_k) +  \eta_k^{-1}(x^{k+1} - x^k) + h^\prime(x^{k+1}) = 0,     
\end{align*}
which along with the convexity of $h$ implies that
\begin{align}
h(x^{k+1}) & \le h(x^*) + h^\prime(x^{k+1})^T(x^{k+1} - x^*) = h(x^*) + G(x^k;\xi_k)^T(x^* - x^{k+1}) + \eta_k^{-1}(x^{k+1} - x^k)^T(x^* - x^{k+1})\nonumber\\
&= h(x^*) + G(x^k;\xi_k)^T(x^* - x^{k+1}) + (2\eta_k)^{-1}(\|x^k-x^*\|^2 - \|x^{k+1}-x^*\|^2 -\|x^{k+1} - x^k\|^2).\label{inter1}
\end{align}
Denote {$f^\prime(x^k) =\E_{\xi}[G(x^k;\xi)\,|\,\xi_{[k-1]}]$}. It follows from \eqref{cond:ac} that  $f^\prime(x^k)\in\partial f(x^k)$. By this, \eqref{ineq:desc_new}, and the convexity of $f$, one has
\begin{align*}
f(x^{k+1}) & \overset{\eqref{ineq:desc_new}}{\le} f(x^k) + f^\prime(x^k)^T(x^{k+1} - x^k) + \frac{L_f+L(\epsilon)}{2}\|x^{k+1} - x^k\|^2 + \frac{\epsilon}{8} + M_f\|x^{k+1} - x^k\| \\
&\le f(x^*) + f^\prime(x^k)^T(x^{k+1} - x^*) + \frac{L_f+L(\epsilon)}{2}\|x^{k+1} - x^k\|^2 + \frac{\epsilon}{8} + M_f\|x^{k+1} - x^k\|.
\end{align*}
Using this, \eqref{def:zetak-deltak-1}, and \eqref{inter1}, we obtain that 
\begin{align}
F(x^{k+1})&\overset{\eqref{def:zetak-deltak-1}\eqref{inter1}}{\le} F(x^*) + \Delta_k + (f^\prime(x^k) - G(x^k;\xi_k))^T(x^{k+1} - x^k) + (2\eta_k)^{-1}(\|x^k-x^*\|^2 - \|x^{k+1}-x^*\|^2)\nonumber\\
&\quad + \Big(\frac{L_f + L(\epsilon)}{2} - \frac{1}{2\eta_k}\Big) \|x^{k+1} - x^k\|^2 + \frac{\epsilon}{8} + M_f\|x^{k+1} - x^k\|\nonumber\\
&\le F(x^*) + \Delta_k + (f^\prime(x^k) - G(x^k;\xi_k))^T(x^{k+1} - x^k) + (2\eta_k)^{-1}(\|x^k-x^*\|^2 - \|x^{k+1}-x^*\|^2)\nonumber\\
&\quad + \Big(\frac{L_f + L(\epsilon)}{2} - \frac{1}{4\eta_k}\Big)\|x^{k+1} - x^k\|^2 + \frac{\epsilon}{8} + M_f^2\eta_k,\label{inter2}
\end{align}
where the last inequality follows from $M_f\|x^{k+1} - x^k\|\le \|x^{k+1} - x^k\|^2/(4\eta_k) + M_f^2\eta_k$. In addition, let $\alpha^\prime=\alpha/(\alpha-1)$. Observe that  $\alpha^\prime \geq 2$ due to $\alpha \in (1,2]$. This together with \eqref{def:diam-h} and $x^{k+1},x^k\in\mathrm{dom}\,h$ implies that $\|x^{k+1} - x^k\|^{\alpha^\prime}\le D_h^{\alpha^\prime-2}\|x^{k+1} - x^k\|^2$.  Using this, \eqref{def:zetak-deltak-1},  and Young's inequality, we have
\begin{align*}
&(f^\prime(x^k) - G(x^k;\xi_k))^T(x^{k+1} - x^k)\le \frac{\Big(\Big(\frac{\alpha^\prime}{8D_h^{\alpha^\prime-2}\eta_k}\Big)^{1/\alpha^\prime}\|x^{k+1} - x^k\|\Big)^{\alpha^\prime}}{\alpha^\prime} + \frac{\Big(\Big(\frac{8D_h^{\alpha^\prime-2}\eta_k}{\alpha^\prime}\Big)^{1/\alpha^\prime}\delta_k\Big)^\alpha}{\alpha}\\[7pt]
&=\frac{\|x^{k+1} - x^k\|^{\alpha^\prime}}{8D_h^{\alpha^\prime-2}\eta_k} + \frac{(8(\alpha-1))^{\alpha-1}D_h^{2-\alpha}\eta_k^{\alpha-1}\delta_k^\alpha}{\alpha^\alpha}\le\frac{\|x^{k+1} - x^k\|^2}{8\eta_k} + \frac{(8(\alpha-1))^{\alpha-1}D_h^{2-\alpha}\eta_k^{\alpha-1}\delta_k^\alpha}{\alpha^\alpha}.
\end{align*}
By this inequality, \eqref{inter2}, and $\eta_k\in\big(0, \frac{1}{4(L_f + L(\epsilon))}\big]$, we obtain that
\begin{align*}
F(x^{k+1}) - F(x^*) & \le (2\eta_k)^{-1}(\|x^k-x^*\|^2 - \|x^{k+1}-x^*\|^2) + M_f^2\eta_k + \Delta_k\\[4pt]
&\quad + \Big(\frac{L_f + L(\epsilon)}{2} - \frac{1}{8\eta_k}\Big)\|x^{k+1} - x^k\|^2 + \frac{(8(\alpha-1))^{\alpha-1}D_h^{2-\alpha}\eta_k^{\alpha-1}\delta_k^\alpha}{\alpha^\alpha} + \frac{\epsilon}{8}\\
& \le (2\eta_k)^{-1}(\|x^k-x^*\|^2 - \|x^{k+1}-x^*\|^2) + M_f^2\eta_k + \Delta_k  + \frac{(8(\alpha-1))^{\alpha-1}D_h^{2-\alpha}\eta_k^{\alpha-1}\delta_k^\alpha}{\alpha^\alpha} + \frac{\epsilon}{8}. 
\end{align*}
Hence, the conclusion \eqref{ineq:desc-rec} holds. 
\end{proof}

We are now ready to provide a proof of Theorem~\ref{thm:cv-psg}.

\begin{proof}[\textbf{Proof of Theorem \ref{thm:cv-psg}}]
Using \eqref{weight-average},  \eqref{ineq:desc-rec}, and the convexity of $F$, we obtain that
\begin{align}
& F(z^K) - F^*\overset{\eqref{weight-average}}{\le} \frac{\sum_{k=0}^{K-1}\eta_kF(x^{k+1})}{\sum_{k=0}^{K-1}\eta_k} - F^*  =\frac{1}{\sum_{k=0}^{K-1}\eta_k}\sum_{k=0}^{K-1}\eta_k(F(x^{k+1}) - F^*)\nonumber\\
&\overset{\eqref{ineq:desc-rec}}{\le} \frac{\|x^0-x^*\|^2}{2\sum_{k=0}^{K-1}\eta_k}  + \frac{(8(\alpha-1))^{\alpha-1}D_h^{2-\alpha}\sum_{k=0}^{K-1}\eta_k^\alpha\delta_k^\alpha}{\alpha^\alpha\sum_{k=0}^{K-1}\eta_k} + \frac{M_f^2\sum_{k=0}^{K-1}\eta_k^2}{\sum_{k=0}^{K-1}\eta_k} + \frac{\sum_{k=0}^{K-1}\eta_k\Delta_k}{\sum_{k=0}^{K-1}\eta_k} + \frac{\epsilon}{8}. \label{pf1-1stineq}
\end{align}

\medskip

{\textit{Proof of statement (i) of Theorem \ref{thm:cv-psg}.}} By \eqref{def:sigma1} and Lemma \ref{lem:tech-alp} with $c=(8(\alpha-1))^{\alpha-1}D_h^{2-\alpha}/\alpha^\alpha$,  $\hat{\eta}=\eta$, and $\varepsilon=\epsilon$, one has
\begin{align}\label{pf-thm1-lem1-ie}
\frac{(8(\alpha-1))^{\alpha-1}D_h^{2-\alpha}}{\alpha^\alpha}\cdot\sigma^\alpha\eta^{\alpha-1} \overset{\eqref{ineq:tech-alp}}{\le} 8(\alpha-1)^2\Big(\frac{\sigma}{\alpha}\Big)^{\frac{\alpha}{\alpha-1}} \Big(\frac{8D_h}{\epsilon}\Big)^{\frac{2-\alpha}{\alpha-1}}\eta + \frac{\epsilon}{8} \overset{\eqref{def:sigma1}}{=} \Lambda(\epsilon)^2 \eta + \frac{\epsilon}{8}.
\end{align}
In addition, recall from \eqref{def:zetak-deltak-1} and Assumption \ref{asp:basic}(c) that $\E_{\xi_k}[\Delta_k\,|\,\xi_{[k-1]}]=0$ and $\E_{\xi_k}[\delta_k^\alpha\,|\,\xi_{[k-1]}]\le\sigma^\alpha$. Using these, \eqref{def:diam-h}, \eqref{def:eta_k-psg-c}, \eqref{ineq:tech-qm}, \eqref{pf-thm1-lem1-ie}, $\eta_k\equiv\eta$ for all $k$, and taking expectation on \eqref{pf1-1stineq} with respect to $\{\xi_k\}_{k=0}^{K-1}$, we obtain that for all $K\ge1$,
\begin{align*}
& \E[F(z^K) - F(x^*)] \overset{\eqref{pf1-1stineq}}{\le} \ \frac{\|x^0-x^*\|^2}{2K\eta} + \frac{(8(\alpha-1))^{\alpha-1}D_h^{2-\alpha}}{\alpha^\alpha}\cdot \sigma^\alpha\eta^{\alpha-1} + M_f^2\eta + \frac{\epsilon}{8}\\
& \overset{\eqref{def:diam-h}\eqref{pf-thm1-lem1-ie}}{\le}\ \frac{D_h^2}{2K\eta} + (M_f^2 + \Lambda(\epsilon)^2)\eta + \frac{\epsilon}{4} \overset{\eqref{def:eta_k-psg-c}}{=}\ \min_{\hat\eta}\bigg\{\frac{D_h^2}{2K\hat\eta} + (M_f^2 + \Lambda(\epsilon)^2)\hat\eta: \hat\eta\in\bigg(0,\frac{1}{4(L_f + L(\epsilon))}\bigg]\bigg\} + \frac{\epsilon}{4}, \\
& \ \overset{\eqref{ineq:tech-qm}}{\le} \ \frac{2D_h^2(L_f + L(\epsilon))}{K} + \sqrt{2}D_h\bigg(\frac{M_f^2 + \Lambda(\epsilon)^2}{K}\bigg)^{\frac12} + \frac{\epsilon}{4} \le\ \frac{2D_h^2(L_f + L(\epsilon))}{K} + \frac{\sqrt{2}D_h(M_f + \Lambda(\epsilon))}{K^{\frac{1}{2}}} + \frac{\epsilon}{4}, 
\end{align*}
where the third inequality follows from \eqref{ineq:tech-qm} with $(a,b,c)=\big(\frac{D_h^2}{2K}, M_f^2 + \Lambda(\epsilon)^2, \frac{1}{4(L_f + L(\epsilon))}\big)$. Then, by this, one can see that $\E[F(z^K) - F^*] \le \epsilon/4 + \epsilon/2 + \epsilon/4=\epsilon$ holds for all $K$ satisfying \eqref{K1}.
Hence, statement (i) of Theorem \ref{thm:cv-psg} holds.

\medskip 
{\textit{Proof of statement (ii) of Theorem \ref{thm:cv-psg}.}} Recall from the definitions of $\Delta_k$ and $\delta_k$ in \eqref{def:zetak-deltak-1} and Assumption \ref{asp:basic}(c) that $\E_{\xi_k}[\Delta_k\,|\,\xi_{[k-1]}]=0$ and $\E_{\xi_k}[\delta_k^\alpha\,|\,\xi_{[k-1]}]\le\sigma^\alpha$. In addition, by $x^k,x^*\in\mathrm{dom}\;h$, \eqref{def:diam-h}, and \eqref{def:zetak-deltak-1}, one has
\begin{align}\label{ineq:upbd-Deltak-ht-sgd}
|\Delta_k| \overset{\eqref{def:zetak-deltak-1}}{=} \big|(G(x^k;\xi_k) - {\E_{\xi}[G(x^k;\xi)\,|\,\xi_{[k-1]}]})^T(x^k - x^*)\big| \le D_h\delta_k\qquad\forall k\ge0,
\end{align}
which along with $\E_{\xi_k}[\delta_k^\alpha\,|\,\xi_{[k-1]}]\le\sigma^\alpha$ implies that $\E_{\xi_k}[|\Delta_k|^\alpha\,|\,\xi_{[k-1]}]\le(\sigma D_h)^\alpha$. Hence, the assumptions of Lemma \ref{lem:hts} hold with $\phi^k=\Delta_k$ and $\varsigma=\sigma D_h$. It then follows from Lemma \ref{lem:hts} with $t=\ln(2/\delta)$ that for any $\delta\in(0,1)$ and $K\ge1$,
\begin{align}\label{ineq:prob-1}
\mathbb{P}\bigg(\frac{1}{K}\sum_{k=0}^{K-1}\Delta_k > \frac{\bar B(\delta)}{K^{\frac{\alpha-1}{\alpha}}}\bigg) \overset{\eqref{def:BC31}}{=} \mathbb{P}\bigg(\frac{1}{K}\sum_{k=0}^{K-1}\Delta_k > \frac{2^{\frac{2-\alpha}{\alpha}}\sigma D_h(2/\delta)^{\frac{1}{\alpha}}}{K^{\frac{\alpha-1}{\alpha}}}\bigg) \le \frac{\delta}{2}.
\end{align}
In addition, notice that $\E[\delta_k^\alpha]=\E_{\xi_{[k-1]}}[\E_{\xi_k}[\delta_k^\alpha\,|\,\xi_{[k-1]}]]\le\sigma^\alpha$ for all $k\ge0$. By this and Markov's inequality, one has that for all $\delta\in(0,1)$ and $K\ge1$,
\begin{align}\label{ineq:prob-ht-a}
\mathbb{P}\bigg(\frac{1}{K}\sum_{k=0}^{K-1}\delta_k^\alpha > \frac{2\sigma^\alpha}{\delta}\bigg) \le \frac{\delta\E[\sum_{k=0}^{K-1}\delta_k^\alpha]}{2 K\sigma^\alpha} \le \frac{\delta}{2}.
\end{align}
In view of \eqref{ineq:prob-1} and \eqref{ineq:prob-ht-a}, we can see that for all $K\ge1$,
\begin{align}\label{ineq:prob-13}
\mathbb{P}\bigg(\bigg\{\frac{1}{K}\sum_{k=0}^{K-1}\Delta_k \le \frac{\bar B(\delta)}{K^{\frac{\alpha-1}{\alpha}}}\bigg\}\bigcap\bigg\{\frac{1}{K}\sum_{k=0}^{K-1}\delta_k^\alpha \le \frac{2\sigma^\alpha}{\delta}\bigg\}\bigg) \ge 1 - \delta.
\end{align}
On the other hand, by \eqref{def:sigma1} and \eqref{ineq:tech-alp} with $c=(8(\alpha-1))^{\alpha-1}D_h^{2-\alpha}(2/\delta)/\alpha^\alpha$, $\hat{\eta}=\bar\eta$, and $\varepsilon=\epsilon$, one has that
\begin{align}\label{pf-thm1-lem1-hp}
\frac{(8(\alpha-1))^{\alpha-1}D_h^{2-\alpha}{(2/\delta)}}{\alpha^\alpha}\cdot\sigma^\alpha\bar\eta^{\alpha-1} &\overset{\eqref{ineq:tech-alp}}{\le} 8(\alpha-1)^2\Big(\frac{\sigma}{\alpha}\Big)^{\frac{\alpha}{\alpha-1}} \Big(\frac{8D_h}{\epsilon}\Big)^{\frac{2-\alpha}{\alpha-1}}{\Big(\frac{2}{\delta}\Big)^{\frac{1}{\alpha-1}}}\bar\eta + \frac{\epsilon}{8}\nonumber\\
&\overset{\eqref{def:sigma1}}{=} \bar\Lambda(\delta,\epsilon)^2\bar\eta + \frac{\epsilon}{8}.
\end{align}
In view of \eqref{def:diam-h}, \eqref{def:eta_k-psg-c-2}, \eqref{ineq:tech-qm}, \eqref{ineq:prob-13},  \eqref{pf-thm1-lem1-hp}, and \eqref{pf1-1stineq} with $\eta_k\equiv\bar\eta$ for all $k$, we have that for all $K\ge 1$, it holds that with probability at least $1-\delta$, 
\begin{align*} 
F(z^K) - F^* \overset{\eqref{pf1-1stineq}}{\le}&\ \frac{\|x^0-x^*\|^2}{2K\bar\eta}  + \frac{(8(\alpha-1))^{\alpha-1}D_h^{2-\alpha}}{K\alpha^\alpha}\cdot \bar\eta^{\alpha-1}\sum_{k=0}^{K-1}\delta_k^\alpha + \frac{1}{K}\sum_{k=0}^{K-1}\Delta_k + M_f^2\bar\eta + \frac{\epsilon}{8}\\
\overset{\eqref{def:diam-h}\eqref{ineq:prob-13}}{\le}&\ \frac{D_h^2}{2K\bar\eta}  + \frac{(8(\alpha-1))^{\alpha-1}D_h^{2-\alpha}{(2/\delta)}}{\alpha^\alpha} \cdot \sigma^\alpha \bar\eta^{\alpha-1} + \frac{\bar B(\delta)}{K^{\frac{\alpha-1}{\alpha}}} + M_f^2\bar\eta + \frac{\epsilon}{8}\\
\overset{\eqref{pf-thm1-lem1-hp}}{\le}&\ \frac{D_h^2}{2K\bar\eta} + (M_f^2 + \bar\Lambda(\delta,\epsilon)^2)\bar\eta + \frac{\bar B(\delta)}{K^{\frac{\alpha-1}{\alpha}}} + \frac{\epsilon}{4}\\
\overset{\eqref{def:eta_k-psg-c-2}}{=}&\ \min_{\hat\eta}\bigg\{\frac{D_h^2}{2K\hat{\eta}} + (M_f^2 + \bar{\Lambda}(\delta,\epsilon)^2)\hat{\eta}: \hat\eta\in\bigg(0,\frac{1}{4(L_f + L(\epsilon))}\bigg]\bigg\} + \frac{\bar{B}(\delta)}{K^{\frac{\alpha-1}{\alpha}}}+ \frac{\epsilon}{4} \\
\overset{\eqref{ineq:tech-qm}}{\le} &\ \frac{2D_h^2(L_f + L(\epsilon))}{K} + \sqrt{2}D_h\Big(\frac{M_f^2 + \bar\Lambda(\delta,\epsilon)^2}{K}\Big)^{\frac{1}{2}} + \frac{\bar B(\delta)}{K^{\frac{\alpha-1}{\alpha}}} + \frac{\epsilon}{4}\\
\le &\ \frac{2D_h^2(L_f + L(\epsilon))}{K} + \frac{\sqrt{2}D_h(M_f + \bar\Lambda(\delta,\epsilon))}{K^{\frac{1}{2}}} + \frac{\bar B(\delta)}{K^{\frac{\alpha-1}{\alpha}}} + \frac{\epsilon}{4},
\end{align*}
where the fourth inequality is due to \eqref{ineq:tech-qm} with $(a,b,c)=\big(\frac{D_h^2}{2K}, M_f^2 + \bar\Lambda(\delta,\epsilon)^2, \frac{1}{4(L_f + L(\epsilon))}\big)$. It then follows that $F(z^K) - F^* \le \epsilon/4+\epsilon/4+\epsilon/4+ \epsilon/4=\epsilon$ holds with probability at least $1-\delta$ for all $K$ satisfying \eqref{K11}. Hence, statement (ii) of Theorem \ref{thm:cv-psg} holds.

\medskip

{\textit{Proof of statement (iii) of Theorem \ref{thm:cv-psg}.}} Recall from the definitions of $\Delta_k$ and $\delta_k$ in \eqref{def:zetak-deltak-1} and Assumption~\ref{asp:sub-weibull} that $\E_{\xi_k}[\Delta_k\,|\,\xi_{[k-1]}]=0$ and $\E_{\xi_k}[\exp\{(\delta_k/\sigma_0)^\beta\}\,|\,\xi_{[k-1]}]\le e$. In addition, by this and \eqref{ineq:upbd-Deltak-ht-sgd}, one has that
\begin{align*}
\E_{\xi_k}\big[\exp\big\{|\Delta_k/(\sigma_0 D_h)|^\beta\big\}\,|\,\xi_{[k-1]}\big]
\le \E_{\xi_k}\big[\exp\{(\delta_k/\sigma_0)^\beta\}\,|\,\xi_{[k-1]}\big] \le e\qquad\forall k\ge0.
\end{align*}
Hence, the assumptions of Lemma \ref{lem:mds} hold with $\phi^k=\Delta_k$ and $\varsigma=\sigma_0 D_h$. It then follows from \eqref{def:BC31} and Lemma \ref{lem:mds} with $t=\ln(2/\delta)$ that for any $\delta\in(0,1)$ and $K\ge1$,
\begin{align}
\mathbb{P}\bigg(\frac{1}{K}\sum_{k=0}^{K-1}\Delta_k > \frac{B_1(\delta)}{K^{\frac{1}{2}}}+\frac{B_2(\delta)}{K^{\min\{\frac{1}{\beta},1\}}}\bigg) \le\frac{\delta}{2}.\label{ineq:prob-133}
\end{align}
Next, define $\mu_k:=\E_{\xi_k}[\delta_k^2\,|\,\xi_{[k-1]}]$ and $\psi^k:=\delta_k^2-\mu_k$ for all $k\ge0$. By Lemma \ref{lem:var-subweibull} and \eqref{def:sigma12}, one has that $\mu_k\le\sigma_{1}^2$ for all $k\ge0$. We next verify that $\{\psi^k\}$ is a martingale difference sequence of sub-Weibull random variables with exponent $\beta/2$. The definitions in \eqref{def:sigma12} imply that $\sigma_2^\beta=\sigma_0^\beta+\sigma_1^\beta$. Since $\beta/2\in(0,1]$, it follows that $|\psi^k|^{\beta/2}\le(\delta_k^2+\mu_k)^{\beta/2}\le\delta_k^\beta+\mu_k^{\beta/2}$. Denote $\theta:=\sigma_0^\beta/\sigma_2^\beta\in(0,1)$. Using these and conditional Jensen's inequality, we obtain that
\begin{align}
\E_{\xi_k}\bigg[\exp\bigg\{\bigg|\frac{\psi^k}{\sigma_2^2}\bigg|^{\frac{\beta}{2}}\bigg\}\,\bigg|\,\xi_{[k-1]}\bigg] & \le \E_{\xi_k}\bigg[\exp\bigg\{\frac{\delta_k^\beta + \mu_k^{\frac{\beta}{2}}}{\sigma_2^{\beta}}\bigg\}\,\bigg|\,\xi_{[k-1]}\bigg] = \exp\bigg\{\frac{\mu_k^{\beta/2}}{\sigma_2^{\beta}}\bigg\}\E_{\xi_k}\bigg[\bigg(\exp\bigg\{\bigg(\frac{\delta_k}{\sigma_0}\bigg)^{\beta}\bigg\}\bigg)^\theta\,\bigg|\,\xi_{[k-1]}\bigg] \nonumber\\
& \le \exp\bigg\{\frac{\sigma_{1}^{\beta}}{\sigma_{2}^{\beta}}\bigg\}\bigg(\E_{\xi_k}\bigg[\exp\bigg\{\bigg(\frac{\delta_k}{\sigma_0}\bigg)^{\beta}\bigg\}\,\bigg|\,\xi_{[k-1]}\bigg]\bigg)^\theta \le \exp\bigg\{\theta + \frac{\sigma_{1}^{\beta}}{\sigma_{2}^{\beta}}\bigg\} = e,\label{ineq:beta/2-subweibull}
\end{align}
Hence, the assumptions of Lemma~\ref{lem:mds} hold with $\phi^k=\psi^k$, $\varsigma=\sigma_2^2$, and $\beta$ replaced by $\beta/2$. It then follows from Lemma \ref{lem:mds} with $t=\ln(2/\delta)$ that for any $\delta\in(0,1)$ and $K\ge1$,
\begin{align*}
\mathbb{P}\bigg(\bigg|\frac{1}{K}\sum_{k=0}^{K-1}\psi^k\bigg| > \frac{\Lambda_1(\delta)^2}{4D_h^2K^{\frac{1}{2}}}+\frac{\Lambda_2(\delta)^2}{4D_h^2K}\bigg)\overset{\eqref{def:C32}}{=}\mathbb{P}\bigg(\bigg|\frac{1}{K}\sum_{k=0}^{K-1}\psi^k\bigg| > 4e^{\frac{3}{2}}\bigg(\frac{8}{\beta}\bigg)^{\frac{2}{\beta}}\sigma_2^2\bigg(\bigg(\frac{\ln(2/\delta)+1}{K}\bigg)^{\frac{1}{2}} + \frac{[\ln(2/\delta)+1]^{\frac{2}{\beta}}}{K^{\min\{\frac{2}{\beta},1\}}}\bigg)\bigg) \le \frac{\delta}{2},
\end{align*}
where the first equality is due to \eqref{def:C32} and $\min\{2/\beta,1\}=1$. Since $\sum_{k=0}^{K-1}\delta_k^2=\sum_{k=0}^{K-1}\psi^k+\sum_{k=0}^{K-1}\mu_k\le\sum_{k=0}^{K-1}\psi^k+K\sigma_1^2$, it follows that
\begin{align}\label{ineq:prob-1333}
\mathbb{P}\bigg(\frac{1}{K}\sum_{k=0}^{K-1}\delta_k^2>\sigma_1^2+\frac{\Lambda_1(\delta)^2}{4D_h^2K^{\frac{1}{2}}}+\frac{\Lambda_2(\delta)^2}{4D_h^2K}\bigg) \le \frac{\delta}{2}.    
\end{align}
In view of \eqref{ineq:prob-133} and \eqref{ineq:prob-1333}, we can see that for all $K\ge1$,
\begin{align}\label{ineq:prob-13333}
\mathbb{P}\bigg(\bigg\{\frac{1}{K}\sum_{k=0}^{K-1}\Delta_k \le {\frac{B_1(\delta)}{K^{\frac{1}{2}}}+\frac{B_2(\delta)}{K^{\min\{\frac{1}{\beta},1\}}}}\bigg\}\bigcap\bigg\{\frac{1}{K}\sum_{k=0}^{K-1}\delta_k^2\le \sigma_1^2 +\frac{\Lambda_1(\delta)^2}{4D_h^2K^{\frac{1}{2}}}+\frac{\Lambda_2(\delta)^2}{4D_h^2K}\bigg\}\bigg)\ge1-\delta.
\end{align}
In view of \eqref{def:diam-h}, \eqref{def:eta_k-psg-c-3}, \eqref{ineq:tech-qm}, \eqref{ineq:prob-13333}, and \eqref{pf1-1stineq} with $\alpha=2$ and $\eta_k\equiv\tilde\eta$ for all $k$, we have that for all $K\ge 1$, it holds that with probability at least $1-\delta$, 
\begin{align} 
& F(z^K) - F^* \overset{\eqref{pf1-1stineq}}{\le} \frac{\|x^0-x^*\|^2}{2K\tilde\eta}  + \frac{2\tilde\eta}{K}\sum_{k=0}^{K-1}\delta_k^2 + \frac{1}{K}\sum_{k=0}^{K-1}\Delta_k + M_f^2\tilde\eta + \frac{\epsilon}{8}\nonumber\\
& \overset{\eqref{def:diam-h}\eqref{ineq:prob-13333}}{\le} \frac{D_h^2}{2K\tilde\eta}  + \bigg(M_f^2 + 2\sigma_1^2 + \frac{\Lambda_1(\delta)^2}{2D_h^2K^{\frac{1}{2}}}+\frac{\Lambda_2(\delta)^2}{2D_h^2K}\bigg)\tilde\eta + \frac{B_1(\delta)}{K^{\frac{1}{2}}}+\frac{B_2(\delta)}{K^{\min\{\frac{1}{\beta},1\}}} + \frac{\epsilon}{8}\nonumber\\
&\overset{\eqref{def:eta_k-psg-c-3}}{=} \min_{\hat\eta}\bigg\{\frac{D_h^2}{2K\hat{\eta}} + \bigg(M_f^2 + 2\sigma_1^2 + \frac{\Lambda_1(\delta)^2}{2D_h^2K^{\frac{1}{2}}}+\frac{\Lambda_2(\delta)^2}{2D_h^2K}\bigg)\hat{\eta}: \hat\eta\in\bigg(0,\frac{1}{4(L_f + L(\epsilon))}\bigg]\bigg\} + \frac{B_1(\delta)}{K^{\frac{1}{2}}}+\frac{B_2(\delta)}{K^{\min\{\frac{1}{\beta},1\}}} + \frac{\epsilon}{8}\nonumber\\
&\overset{\eqref{ineq:tech-qm}}{\le} \frac{2D_h^2(L_f + L(\epsilon))}{K} + \sqrt{2}D_h\Big(\frac{M_f^2 + 2\sigma_1^2}{K} + \frac{\Lambda_1(\delta)^2}{2D_h^2K^{\frac{3}{2}}}+\frac{\Lambda_2(\delta)^2}{2D_h^2K^2}\Big)^{\frac{1}{2}} + \frac{B_1(\delta)}{K^{\frac{1}{2}}}+\frac{B_2(\delta)}{K^{\min\{\frac{1}{\beta},1\}}} + \frac{\epsilon}{8}\nonumber\\
&\le \frac{2D_h^2(L_f + L(\epsilon))}{K} + \frac{2D_h(M_f + \sigma_1) + B_1(\delta)}{K^{\frac{1}{2}}} + \frac{\Lambda_1(\delta)}{K^{\frac{3}{4}}} + \frac{B_2(\delta)}{K^{\min\{\frac{1}{\beta},1\}}} + \frac{\Lambda_2(\delta)}{K}  + \frac{\epsilon}{8}.\label{ineq:fun-val-gap-case3-sgd}
\end{align}
where the third inequality uses \eqref{ineq:tech-qm} with
\[
(a,b,c)=\bigg(\frac{D_h^2}{2K}, M_f^2 + 2\sigma_1^2 + \frac{\Lambda_1(\delta)^2}{2D_h^2K^{\frac{1}{2}}}+\frac{\Lambda_2(\delta)^2}{2D_h^2K}, \frac{1}{4(L_f + L(\epsilon))}\bigg).
\]
Therefore, it follows from \eqref{ineq:fun-val-gap-case3-sgd} that $F(z^K)-F^*\le\epsilon/4+\epsilon/4+\epsilon/8 + \epsilon/8 +\epsilon/8+\epsilon/8=\epsilon$ holds with probability at least $1-\delta$ for all $K$ satisfying \eqref{K2}. Hence, statement (iii) of Theorem \ref{thm:cv-psg} holds.

\end{proof}

\subsection{Proof of the main results in Section \ref{sec:apsg}}\label{subsec:pf-apsg}

In this subsection, we first establish a lemma and then use it to prove Theorem~\ref{thm:cv-apsg}.

\begin{lemma}\label{lem:rec-apsg}
Suppose that Assumption~\ref{asp:basic} holds. Let $\epsilon\in(0,1)$ be arbitrarily chosen, $L_f$, $M_f$, $\alpha$, and $\sigma$ be given in Assumption \ref{asp:basic}, and $L(\cdot)$ be defined in \eqref{def:inexact-Lip}. Let $\{(x^k,y^k,z^k)\}$ be the sequence generated by Algorithm~\ref{alg:ac-dq-f} with input parameters $\{(\eta_k,\gamma_k)\}$ satisfying $4(L_f+L(\epsilon\gamma_k))\eta_k\gamma_k\le1$ and $\eta_{k+1}(\gamma_{k+1}^{-1} - 1)\le\eta_k\gamma_k^{-1}$ for all $k=0,\ldots,K-1$. Then, it holds that for all $k=0,\ldots,K-1$,    
\begin{align}
\eta_{k+1}(\gamma_{k+1}^{-1} - 1) (F(z^{k+1}) - F^*) \le&\ \eta_k(\gamma_k^{-1} - 1)(F(z^k) - F^*)  +  (\|x^k-x^*\|^2 - \|x^{k+1}-x^*\|^2)/2\nonumber\\[3pt]
&\ + M_f^2\eta_k^2 + \eta_k \Delta_k + \frac{(8(\alpha-1))^{\alpha-1}D_h^{2-\alpha}\delta_k^\alpha\eta_k^{\alpha}}{\alpha^\alpha} + \frac{\epsilon\eta_k}{8},\label{ineq:rec-apsg}
\end{align}
where 
{
\begin{align}\label{def:zetak-deltak-2}
\Delta_k=(\E_{\xi}[G(y^k;\xi)\,|\,\xi_{[k-1]}]-G(y^k;\xi_k))^T(x^k-x^*),\quad \delta_k=
\|\E_{\xi}[G(y^k;\xi)\,|\,\xi_{[k-1]}]-G(y^k;\xi_k)\|\qquad\forall k\ge0,    
\end{align}
where $\xi,\xi_0,\xi_1,\ldots$ are i.i.d.\ random variables, $\xi_{[-1]}=\emptyset$, and $\xi_{[k]}=\{\xi_t\}_{t=0}^k$ for all $k\ge0$.}
\end{lemma}

\begin{proof}
Fix any $k\in\{0,\ldots,K-1\}$. Using \eqref{a-prox-step} and similar arguments as for deriving \eqref{inter1} with $G(x^k;\xi_k)$ replaced by $G(y^k;\xi_k)$, we can deduce that
\begin{align*}
h(x^{k+1})\le h(x^*) + G(y^k;\xi_k)^T(x^* -x^{k+1}) + (2\eta_k)^{-1}(\|x^k-x^*\|^2 - \|x^{k+1}-x^*\|^2 - \|x^k - x^{k+1}\|^2).
\end{align*}
By this, \eqref{ag-update}, and the convexity of $h$, one has
\begin{align}
\eta_k\gamma_k^{-1} h(z^{k+1}) \le&\ \eta_k(\gamma_k^{-1}-1) h(z^k) + \eta_k h(x^{k+1})\nonumber\\[3pt]   
\le&\ \eta_k(\gamma_k^{-1}-1) h(z^k) + \eta_k h(x^*) + (\|x^k-x^*\|^2 - \|x^{k+1}-x^*\|^2 - \|x^k - x^{k+1}\|^2)/2\nonumber \\[3pt]
&\ + \eta_kG(y^k;\xi_k)^T(x^* -x^{k+1}).\label{a-inter1}
\end{align}
In addition, notice from \eqref{mid-update} and \eqref{ag-update} that $z^{k+1} - y^k = \gamma_k(x^{k+1} - x^k)$. Denote {$f^\prime(y^k) =\E_{\xi}[G(y^k;\xi)\,|\,\xi_{[k-1]}]$}. It follows from \eqref{cond:ac} that  $f^\prime(y^k)\in\partial f(y^k)$. By these and \eqref{ineq:desc_new} with $(y,x)=(z^{k+1}, y^k)$, one has that
\begin{align}
& \eta_k\gamma_k^{-1}f(z^{k+1})\le \ \eta_k\gamma_k^{-1}\Big(f(y^k) + f^\prime(y^k)^T(z^{k+1} - y^k) + \frac{L_f + L(\epsilon\gamma_k)}{2}\|z^{k+1} - y^k\|^2 + \frac{\epsilon\gamma_k}{8} + M_f\|z^{k+1} - y^k\|\Big) \nonumber\\
&= \eta_k\gamma_k^{-1}(f(y^k) + f^\prime(y^k)^T(z^{k+1} - y^k))+ \frac{(L_f + L(\epsilon\gamma_k))\eta_k\gamma_k}{2}\|x^{k+1} - x^k\|^2 + \frac{\epsilon\eta_k}{8} + M_f\eta_k\|x^{k+1} - x^k\|.\label{a-inter2}
\end{align}
Also, it follows from \eqref{ag-update} and the convexity of $f$ that 
\begin{align*}
\eta_k\gamma_k^{-1}(f(y^k) + f^\prime(y^k)^T(z^{k+1} - y^k)) \overset{\eqref{ag-update}}{=} &\ \eta_k(\gamma_k^{-1} - 1)(f(y^k) + f^\prime(y^k)^T(z^k - y^k)) \\
& + \eta_k (f(y^k) + f^\prime(y^k)^T(x^{k+1} - y^k))\\
\le &\ \eta_k(\gamma_k^{-1} - 1)f(z^k) + \eta_k (f(y^k) + f^\prime(y^k)^T(x^{k+1} - y^k))\\
=&\ \eta_k(\gamma_k^{-1} - 1)f(z^k) + \eta_k (f(y^k) + f^\prime(y^k)^T(x^* - y^k))\\
& + \eta_k f^\prime(y^k)^T(x^{k+1} - x^*)\\
\le&\ \eta_k(\gamma_k^{-1} - 1)f(z^k) + \eta_k f(x^*) + \eta_k f^\prime(y^k)^T(x^{k+1} - x^*),
\end{align*}
where the first and second inequalities are due to the convexity of $f$. This along with \eqref{a-inter2} implies that
\begin{align}
\eta_k\gamma_k^{-1} f(z^{k+1})&\overset{\eqref{a-inter2}}{\le} \eta_k(\gamma_k^{-1} - 1)f(z^k) + \eta_k f(x^*) + \eta_k f^\prime(y^k)^T(x^{k+1} - x^*)\nonumber\\[5pt]
&\qquad + \frac{(L_f + L(\epsilon\gamma_k))\eta_k\gamma_k}{2}\|x^{k+1} - x^k\|^2 + \frac{\epsilon\eta_k}{8} + M_f\eta_k\|x^{k+1} - x^k\|.\nonumber
\end{align}
By this, \eqref{ag-update} and \eqref{a-inter1}, one has that
\begin{align}
\eta_k\gamma_k^{-1} F(z^{k+1}) 
\le&\ \eta_k(\gamma_k^{-1} - 1)F(z^k) + \eta_k F(x^*) + (\|x^k-x^*\|^2 - \|x^{k+1}-x^*\|^2)/2 + M_f\eta_k\|x^{k+1} - x^k\|\nonumber\\[5pt]
& +\Big(\frac{(L_f + L(\epsilon\gamma_k))\eta_k\gamma_k}{2}- \frac{1}{2}\Big)\|x^{k+1} - x^k\|^2 + \eta_k (f^\prime(y^k) - G(y^k;\xi_k))^T(x^{k+1} - x^*) + \frac{\epsilon\eta_k}{8}\nonumber\\[5pt]
\le&\ \eta_k(\gamma_k^{-1} - 1)F(z^k) + \eta_k F(x^*)  + (\|x^k-x^*\|^2 - \|x^{k+1}-x^*\|^2)/2 + \eta_k \Delta_k + M_f^2\eta_k^2\nonumber\\[5pt]
& +\Big(\frac{(L_f + L(\epsilon\gamma_k))\eta_k\gamma_k}{2}- \frac{1}{4}\Big)\|x^{k+1} - x^k\|^2 + \eta_k (f^\prime(y^k) - G(y^k;\xi_k))^T(x^{k+1} - x^k) + \frac{\epsilon\eta_k}{8},\label{a-inter4}
\end{align}
where the last inequality is due to \eqref{def:zetak-deltak-2} and $M_f\eta_k\|x^{k+1} - x^k\|\le \|x^{k+1} - x^k\|^2/4 + M_f^2\eta_k^2$. In addition, let $\alpha^\prime=\alpha/(\alpha-1)$. Observe that  $\alpha^\prime \geq 2$ due to $\alpha \in (1,2]$. This together with \eqref{def:diam-h} and $x^{k+1},x^k\in\mathrm{dom}\,h$ implies that $\|x^{k+1} - x^k\|^{\alpha^\prime}\le D_h^{\alpha^\prime-2}\|x^{k+1} - x^k\|^2$.  Using this,  \eqref{def:zetak-deltak-2}, and Young's inequality, we obtain that
\begin{align*}
&(f^\prime(y^k) - G(y^k;\xi_k))^T(x^{k+1} - x^k)\le \frac{\Big(\Big(\frac{\alpha^\prime}{8D_h^{\alpha^\prime-2}\eta_k}\Big)^{1/\alpha^\prime}\|x^{k+1} - x^k\|\Big)^{\alpha^\prime}}{\alpha^\prime} + \frac{\Big(\Big(\frac{8D_h^{\alpha^\prime-2}\eta_k}{\alpha^\prime}\Big)^{1/\alpha^\prime}\delta_k\Big)^\alpha}{\alpha}\\[5pt]
&=\frac{\|x^{k+1} - x^k\|^{\alpha^\prime}}{8D_h^{\alpha^\prime-2}\eta_k} + \frac{(8(\alpha-1))^{\alpha-1}D_h^{2-\alpha}\delta_k^\alpha\eta_k^{\alpha-1}}{\alpha^\alpha} \le \frac{\|x^{k+1} - x^k\|^2}{8\eta_k} + \frac{(8(\alpha-1))^{\alpha-1}D_h^{2-\alpha}\delta_k^\alpha\eta_k^{\alpha-1}}{\alpha^\alpha}.
\end{align*}
This along with \eqref{a-inter4} and $4(L_f + L(\epsilon\gamma_k))\eta_k\gamma_k\le1$ implies that
\begin{align*}
\eta_k\gamma_k^{-1} (F(z^{k+1}) - F^*) \le&\ \eta_k(\gamma_k^{-1} - 1)(F(z^k) - F^*)  + (\|x^k-x^*\|^2 - \|x^{k+1}-x^*\|^2)/2  + \eta_k \Delta_k + M_f^2\eta_k^2\\[3pt]
&  + \Big(\frac{(L_f + L(\epsilon\gamma_k))\eta_k\gamma_k}{2} - \frac{1}{8}\Big)\|x^{k+1} - x^k\|^2 + \frac{(8(\alpha-1))^{\alpha-1}D_h^{2-\alpha}\delta_k^\alpha\eta_k^{\alpha}}{\alpha^\alpha} + \frac{\epsilon\eta_k}{8}\\[3pt]
\le&\ \eta_k(\gamma_k^{-1} - 1)(F(z^k) - F^*)  + (\|x^k-x^*\|^2 - \|x^{k+1}-x^*\|^2)/2 + \eta_k \Delta_k + M_f^2\eta_k^2 \\[3pt]
& + \frac{(8(\alpha-1))^{\alpha-1}D_h^{2-\alpha}\delta_k^\alpha\eta_k^{\alpha}}{\alpha^\alpha}  + \frac{\epsilon\eta_k}{8}.
\end{align*}
This inequality together with $\eta_{k+1}(\gamma_{k+1}^{-1} - 1)\le\eta_k\gamma_k^{-1}$ implies that the conclusion \eqref{ineq:rec-apsg} holds. 
\end{proof}

We are now ready to provide a proof of Theorem~\ref{thm:cv-apsg}.

\begin{proof}[\textbf{Proof of Theorem \ref{thm:cv-apsg}}]
Summing up  \eqref{ineq:rec-apsg} over $k=0,\ldots,K-1$, and rearranging terms, we obtain that 
\begin{align}
\eta_K(\gamma_K^{-1} - 1) (F(z^K) - F^*) & \le \eta_0(\gamma_0^{-1} - 1)(F(z^0) - F^*)  + \frac{\|x^0-x^*\|^2}{2}  + \sum_{k=0}^{K-1}\eta_k \Delta_k + M_f^2\sum_{k=0}^{K-1}\eta_k^2\nonumber\\
&\quad  + \frac{(8(\alpha-1))^{\alpha-1}D_h^{2-\alpha}}{\alpha^\alpha}\sum_{k=0}^{K-1}\delta_k^\alpha\eta_k^{\alpha} + \frac{\epsilon}{8}\sum_{k=0}^{K-1}\eta_k. \label{pf2-ineq1}   
\end{align}

\medskip

{\textit{Proof of statement (i) of Theorem \ref{thm:cv-apsg}.}} Recall that $\gamma_k=2/(k+2)$ and $\eta_k=(k+2)\eta/2$ for all $k\ge0$. Using these, the expression of $\eta$ in \eqref{def:eta-apsg-c}, and the fact that $L(\cdot)$ is nonincreasing, we obtain that $4(L_f + L(\epsilon\gamma_k))\eta_k\gamma_k = 4(L_f + L(\epsilon\gamma_k))\eta\le 4(L_f + L(\epsilon/K))\eta\le 1$ holds for all $0\le k\le K$ and $K\ge2$. Also, one can verify that
\begin{align*}
\eta_{k+1}(\gamma_{k+1}^{-1}-1) & =\eta\cdot\frac{k+3}{2}\cdot\frac{k+1}{2}=\frac{\eta(k+3)(k+1)}{4} \le \frac{\eta(k+2)^2}{4}  = \eta_k\gamma_k^{-1}\qquad\forall k\ge0.
\end{align*}
Hence, the assumptions of Lemma \ref{lem:rec-apsg} hold for $(\eta_k,\gamma_k)$ given in statement (i) of Theorem \ref{thm:cv-apsg}. Also, by \eqref{def:sigma1} and \eqref{ineq:tech-alp} with $c=(8(\alpha-1))^{\alpha-1}D_h^{2-\alpha}/\alpha^\alpha$, $\hat{\eta}=\eta_k$, and $\varepsilon=\epsilon$, one has
\begin{align}\label{pf-thm2-lem1-ie}
\frac{(8(\alpha-1))^{\alpha-1}D_h^{2-\alpha}}{\alpha^\alpha}\cdot\sigma^\alpha\eta_k^{\alpha-1} \overset{\eqref{ineq:tech-alp}}{\le} 8(\alpha-1)^2\Big(\frac{\sigma}{\alpha}\Big)^{\frac{\alpha}{\alpha-1}} \Big(\frac{8D_h}{\epsilon}\Big)^{\frac{2-\alpha}{\alpha-1}}\eta_k + \frac{\epsilon}{8} \overset{\eqref{def:sigma1}}{=} \Lambda(\epsilon)^2\eta_k + \frac{\epsilon}{8}\qquad\forall k\ge0.
\end{align}
Using this, $\gamma_0=1$, \eqref{def:diam-h}, \eqref{pf-thm2-lem1-ie},  Assumption \ref{asp:basic}(c), and taking expectation on \eqref{pf2-ineq1} with respect to $\{\xi_k\}_{k=0}^{K-1}$, we obtain that for all $K\ge2$,
\begin{align*}
\eta_K(\gamma_K^{-1}-1)\E[F(z^K) - F^*] \le &\ \eta_0(\gamma_0^{-1} - 1)(F(z^0) - F^*) + \frac{\|x^0-x^*\|^2}{2} + M_f^2\sum_{k=0}^{K-1}\eta_k^2\\
&+ \frac{(8(\alpha-1))^{\alpha-1}D_h^{2-\alpha}}{\alpha^\alpha}\cdot \sigma^\alpha\sum_{k=0}^{K-1}\eta_k^{\alpha} + \frac{\epsilon}{8}\sum_{k=0}^{K-1}\eta_k\\
 \overset{\eqref{def:diam-h}\eqref{pf-thm2-lem1-ie}}{\le}&\ \eta_0(\gamma_0^{-1} - 1)(F(z^0) - F^*) + \frac{D_h^2}{2}  + (M_f^2+\Lambda(\epsilon)^2)\sum_{k=0}^{K-1}\eta_k^2 + \frac{\epsilon}{4}\sum_{k=0}^{K-1}\eta_k\\
= &\ \frac{D_h^2}{2}  + (M_f^2+\Lambda(\epsilon)^2)\sum_{k=0}^{K-1}\eta_k^2 + \frac{\epsilon}{4}\sum_{k=0}^{K-1}\eta_k,
\end{align*}
where the last equality is due to $\gamma_0=1$. Further, using this, \eqref{def:eta-apsg-c}, \eqref{ineq:tech-qm}, $\gamma_k=2/(k+2)$, $\eta_k=(k+2)\eta/2$, and rearranging the terms, we obtain that for all $K\ge2$,
\begin{align*}
\E[F(z^K) - F^*] \le &\ \frac{2D_h^2}{(K+2)K \eta} + \frac{(M_f^2 + \Lambda(\epsilon)^2)\eta}{(K+2)K}\sum_{k=0}^{K-1}(k+2)^2 + \frac{\epsilon}{2(K+2)K} \sum_{k=0}^{K-1}(k+2) \\
=&\ \frac{2D_h^2}{(K+2)K\eta} + \frac{(M_f^2 + \Lambda(\epsilon)^2)((K+1)(K+2)(2K+3)/6-1)\eta}{(K+2)K} + \frac{\epsilon((K+1)(K+2)/2-1)}{2(K+2)K} \\
\le&\ \frac{2D_h^2}{(K+2)K\eta} + \frac{(M_f^2 + \Lambda(\epsilon)^2)(2K+3)\eta}{3} + \frac{\epsilon}{2}\\
\overset{\eqref{def:eta-apsg-c}}{=}&\ \min_{\hat\eta}\bigg\{\frac{2D_h^2}{(K+2)K \hat\eta} + \frac{(M_f^2 + \Lambda(\epsilon)^2)(2K+3)\hat\eta}{3}: \hat\eta\in\bigg(0,\frac{1}{4(L_f + L(\epsilon/K))}\bigg]\bigg\} + \frac{\epsilon}{2}, \\
\overset{\eqref{ineq:tech-qm}}{\le}&\ \frac{8D_h^2(L_f + L(\epsilon/K))}{(K+2)K} + 4D_h\bigg(\frac{M_f^2 + \Lambda(\epsilon)^2}{3K}\bigg)^{\frac{1}{2}} + \frac{\epsilon}{2} \\
{\le}&\ \frac{8D_h^2L_f}{K^2} + \frac{8D_h^2L(\epsilon)}{K^{\frac{1+3\nu}{1+\nu}}} + \frac{4D_h(M_f + \Lambda(\epsilon))}{\sqrt{3}K^{\frac{1}{2}}} + \frac{\epsilon}{2},
\end{align*}
where the third inequality follows from  \eqref{ineq:tech-qm} with $(a,b,c)=\big(\frac{2D_h^2}{(K+2)K}, \frac{(M_f^2 + \Lambda(\epsilon)^2)(2K+3)}{3}, \frac{1}{4(L_f+L(\epsilon/K))}\big)$, and the last inequality is due to \eqref{def:inexact-Lip} and $K\ge2$. It then follows that $\E[F(z^K) - F^*] \le \epsilon/6+\epsilon/6+\epsilon/6 + \epsilon/2=\epsilon$ holds for all $K$ satisfying \eqref{K3}.
Hence, statement (i) of Theorem \ref{thm:cv-apsg} holds.

\medskip

{\textit{Proof of statement (ii) of Theorem \ref{thm:cv-apsg}.}} Similar to the proof of statement (i), one can show that the assumptions of Lemma \ref{lem:rec-apsg} hold for $(\eta_k,\gamma_k)$ defined in statement (ii) of Theorem \ref{thm:cv-apsg}. Recall from the expressions of $\Delta_k$ and $\delta_k$ in \eqref{def:zetak-deltak-2} and Assumption \ref{asp:basic}(c) that $\E_{\xi_k}[\Delta_k\,|\,\xi_{[k-1]}]=0$ and $\E_{\xi_k}[\delta_k^\alpha\,|\,\xi_{[k-1]}]\le\sigma^\alpha$. In addition, by $x^k,x^*\in\mathrm{dom}\;h$, \eqref{def:diam-h}, \eqref{def:zetak-deltak-2}, and $\eta_k=(k+2)\bar\eta/2$, one has that for all $0\le k\le K$,
\begin{align}
|\eta_k\Delta_k| \overset{\eqref{def:zetak-deltak-2}}{=} |\eta_k(G(y^k;\xi_k) - {\E_{\xi}[G(y^k;\xi)\,|\,\xi_{[k-1]}]})^T(x^k - x^*)| \le \eta_K D_h\delta_k,\label{ineq:aspg-etakDeltak}
\end{align}
which along with $\E_{\xi_k}[\delta_k^\alpha\,|\,\xi_{[k-1]}]\le\sigma^\alpha$ implies that $\E_{\xi_k}[|\eta_k\Delta_k|^\alpha\,|\,\xi_{[k-1]}]\le(\sigma\eta_KD_h)^\alpha$.
Hence, the assumptions of Lemma \ref{lem:hts} hold with $\phi^k=\eta_k\Delta_k$ and $\varsigma=\sigma\eta_K D_h$. 
Using these and similar arguments as for proving \eqref{ineq:prob-13}, we obtain that for all $K\ge2$, it holds that
\begin{align}\label{ineq:prob-23}
\mathbb{P}\bigg(\bigg\{\frac{1}{K}\sum_{k=0}^{K-1}\eta_k\Delta_k \le \frac{\bar B(\delta)\eta_K}{K^{\frac{\alpha-1}{\alpha}}}\bigg\}\bigcap\bigg\{\frac{1}{K}\sum_{k=0}^{K-1}\delta_k^\alpha \le \frac{2\sigma^\alpha}{\delta}\bigg\}\bigg) \ge 1 - \delta.
\end{align}
On the other hand, by \eqref{def:sigma1} and  \eqref{ineq:tech-alp} with $c=(8(\alpha-1))^{\alpha-1}D_h^{2-\alpha}(2/\delta)/\alpha^\alpha$, $\hat{\eta}=\eta_K$, and $\varepsilon=\epsilon$, one has that
\begin{align}\label{pf-thm2-lem1-hp}
\frac{(8(\alpha-1))^{\alpha-1}D_h^{2-\alpha}{(2/\delta)}}{\alpha^\alpha}\cdot\sigma^\alpha\eta_K^{\alpha-1} &\overset{\eqref{ineq:tech-alp}}{\le} 8(\alpha-1)^2\Big(\frac{\sigma}{\alpha}\Big)^{\frac{\alpha}{\alpha-1}} \Big(\frac{8D_h}{\epsilon}\Big)^{\frac{2-\alpha}{\alpha-1}}{\Big(\frac{2}{\delta}\Big)^{\frac{1}{\alpha-1}}}\eta_K + \frac{\epsilon}{8}\nonumber\\
&\overset{\eqref{def:sigma1}}{=} \bar\Lambda(\delta,\epsilon)^2\eta_K + \frac{\epsilon}{8}.
\end{align}
Using these, \eqref{pf2-ineq1}, $\gamma_0=1$, $\gamma_k=2/(k+2)$, and $\eta_k=(k+2)\bar\eta/2$, we obtain that for all $K\ge2$, it holds that with probability at least $1-\delta$,
\begin{align*} 
&\bar\eta(K+2)K (F(z^K) - F^*)/4 = \eta_K(\gamma_K^{-1} - 1) (F(z^K) - F^*)\\
&\overset{\eqref{pf2-ineq1}}{\le} \eta_0(\gamma_0^{-1} - 1)(F(z^0) - F^*)  + \frac{D_h^2}{2}   + \frac{(8(\alpha-1))^{\alpha-1}D_h^{2-\alpha}}{\alpha^\alpha}\eta_K^{\alpha}\sum_{k=0}^{K-1}\delta_k^\alpha + M_f^2\sum_{k=0}^{K-1}\eta_k^2 + \sum_{k=0}^{K-1}\eta_k \Delta_k + \frac{\epsilon}{8}\sum_{k=0}^{K-1}\eta_k \\
&\overset{\eqref{ineq:prob-23}}{\le}\eta_0(\gamma_0^{-1} - 1)(F(z^0) - F^*)  + \frac{D_h^2}{2} + \frac{(8(\alpha-1))^{\alpha-1}D_h^{2-\alpha}{(2/\delta)}}{\alpha^\alpha}\sigma^\alpha\eta_K^{\alpha}K + M_f^2\sum_{k=0}^{K-1}\eta_k^2+\bar B(\delta)K^{\frac{1}{\alpha}}\eta_K + \frac{\epsilon}{8} \sum_{k=0}^{K-1}\eta_k\nonumber\\ 
&\overset{\eqref{pf-thm2-lem1-hp}}{\le} \eta_0(\gamma_0^{-1} - 1)(F(z^0) - F^*) + \frac{D_h^2}{2} + (M_f^2 + \bar\Lambda(\delta,\epsilon)^2) K \eta_K^2   + \frac{\bar B(\delta)}{K^{\frac{\alpha-1}{\alpha}}}\cdot K\eta_K + \frac{\epsilon}{8}\Big(\sum_{k=0}^{K-1}\eta_k + K\eta_K\Big) \\
&\le  \frac{D_h^2}{2} + (M_f^2 + \bar\Lambda(\delta,\epsilon)^2)\Big(\frac{K+2}{2}\Big)^2K \bar\eta^2 + \frac{\bar{B}(\delta)}{K^{\frac{\alpha-1}{\alpha}}}\cdot \frac{(K+2)K}{2}\bar\eta + \frac{\epsilon(K+2)K\bar\eta}{8},
\end{align*}
where the last inequality is due to $\gamma_0=1$ and $\eta_k=(k+2)\bar\eta/2$ for all $k\ge0$. Further,  
by \eqref{def:eta-apsg-c-1}, \eqref{ineq:tech-qm}, and rearranging the terms in the above inequality, one has that for all $K\ge2$, it holds that with probability at least $1-\delta$,
\begin{align*}
F(z^K) - F^* \le&\ \frac{2D_h^2}{(K+2)K\bar\eta} + (M_f^2+\bar\Lambda(\delta,\epsilon)^2)(K+2)\bar\eta + \frac{2\bar B(\delta)}{K^{\frac{\alpha-1}{\alpha}}} + \frac{\epsilon}{2}\\
\overset{\eqref{def:eta-apsg-c-1}}{=}&\ \min_{\hat\eta}\bigg\{\frac{2D_h^2}{(K+2)K \hat\eta} + (M_f^2+\bar\Lambda(\delta,\epsilon)^2)(K+2)\hat\eta: \hat\eta\in\bigg(0,\frac{1}{4(L_f + L(\epsilon/K))}\bigg]\bigg\} + \frac{2\bar B(\delta)}{K^{\frac{\alpha-1}{\alpha}}} + \frac{\epsilon}{2}\\
\le&\ \frac{8D_h^2(L_f + L(\epsilon/K))}{(K+2)K} + 2\sqrt{2}D_h\Big(\frac{M_f^2 + \bar\Lambda(\delta,\epsilon)^2}{K}\Big)^{\frac{1}{2}} + \frac{2\bar B(\delta)}{K^{\frac{\alpha-1}{\alpha}}} + \frac{\epsilon}{2}\\
\le&\ \frac{8D_h^2L_f}{K^2} + \frac{8D_h^2L(\epsilon)}{K^{\frac{1+3\nu}{1+\nu}}} + \frac{2\sqrt{2}D_h(M_f + \bar\Lambda(\delta,\epsilon))}{K^{\frac{1}{2}}} + \frac{2\bar B(\delta)}{K^{\frac{\alpha-1}{\alpha}}} + \frac{\epsilon}{2},
\end{align*}
where the second inequality is due to \eqref{ineq:tech-qm} with $(a,b,c)=\big(\frac{2D_h^2}{(K+2)K},(M_f^2 + \bar\Lambda(\delta,\epsilon)^2)(K+2), \frac{1}{4(L_f + L(\epsilon/K))}\big)$. It then follows that $F(z^K) - F^* \le \epsilon/8+\epsilon/8+\epsilon/8+\epsilon/8 + \epsilon/2=\epsilon$ holds with probability at least $1-\delta$ for all $K$ satisfying \eqref{K33}.
Hence, statement (ii) of Theorem \ref{thm:cv-apsg} holds.

\medskip

{\textit{Proof of statement (iii) of Theorem \ref{thm:cv-apsg}.}} Similar to the proof of statement (i), we can show that the assumptions of Lemma \ref{lem:rec-apsg} hold for $(\eta_k,\gamma_k)$ defined in statement (iii) of Theorem \ref{thm:cv-apsg}. By \eqref{ineq:aspg-etakDeltak} and Assumption \ref{asp:sub-weibull}, one has that
\[
\E_{\xi_k}\big[\exp\{|\eta_k\Delta_k/(\sigma_0\eta_KD_h)|^\beta\}\,\big|\,\xi_{[k-1]}\big]   \le \E_{\xi_k}\big[\exp\{(\delta_k/\sigma_0)^\beta\}\,\big|\,\xi_{[k-1]}\big] \le e\quad\forall 0\le k\le K.
\]
Hence, the assumptions of Lemma \ref{lem:mds} hold with $\phi^k=\eta_k\Delta_k$ and $\varsigma=\sigma_0\eta_K D_h$. In addition, define $\mu_k:=\E_{\xi_k}[\delta_k^2\,|\,\xi_{[k-1]}]$ and $\psi^k:=\delta_k^2-\mu_k$ for all $k\ge0$. By Lemma \ref{lem:var-subweibull} and \eqref{def:sigma12}, one has that $\mu_k\le\sigma_{1}^2$ for all $k\ge0$. By the same arguments as for \eqref{ineq:beta/2-subweibull}, one can show that $\{\psi^k\}$ is a martingale difference sequence of sub-Weibull random variables with exponent $\beta/2$. Using similar arguments as for \eqref{ineq:prob-13333}, we obtain that for all $K\ge1$,
\begin{align}\label{ineq:sub-weibull-233}
\mathbb{P}\bigg(\bigg\{\frac{1}{K}\sum_{k=0}^{K-1}\eta_k\Delta_k \le {\frac{B_1(\delta)\eta_K}{K^{\frac{1}{2}}}+\frac{B_2(\delta)\eta_K}{K^{\min\{\frac{1}{\beta},1\}}}}\bigg\}\bigcap\bigg\{\frac{1}{K}\sum_{k=0}^{K-1}\delta_k^2\le \sigma_1^2 +\frac{\Lambda_1(\delta)^2}{4D_h^2K^{\frac{1}{2}}}+\frac{\Lambda_2(\delta)^2}{4D_h^2K}\bigg\}\bigg)\ge1-\delta.
\end{align}
Using these, \eqref{pf2-ineq1} with $\alpha=2$, $\gamma_0=1$, $\gamma_k=2/(k+2)$, and $\eta_k=(k+2)\tilde\eta/2$, we obtain that for all $K\ge2$, it holds that with probability at least $1-\delta$,
\begin{align*} 
&(K+2)K\tilde\eta(F(z^K) - F^*)/4 = \eta_K(\gamma_K^{-1} - 1) (F(z^K) - F^*)\\
&\overset{\eqref{pf2-ineq1}}{\le} \eta_0(\gamma_0^{-1} - 1)(F(z^0) - F^*)  + \frac{D_h^2}{2}   + 2\eta_K^2\sum_{k=0}^{K-1}\delta_k^2 + M_f^2\sum_{k=0}^{K-1}\eta_k^2 + \sum_{k=0}^{K-1}\eta_k \Delta_k + \frac{\epsilon}{8}\sum_{k=0}^{K-1}\eta_k \\
&\overset{\eqref{ineq:sub-weibull-233}}{\le}
\eta_0(\gamma_0^{-1} - 1)(F(z^0) - F^*) + \frac{D_h^2}{2} + (M_f^2 + 2\sigma_1^2) K \eta_K^2   + \frac{\Lambda_1(\delta)^2K^{\frac{1}{2}}\eta_K^2}{2D_h^2}+\frac{\Lambda_2(\delta)^2\eta_K^2}{2D_h^2} \\
&\qquad + \Big(\frac{B_1(\delta)}{K^{\frac{1}{2}}} + \frac{B_2(\delta)}{K^{\min\{\frac{1}{\beta},1\}}}\Big)K\eta_K + \frac{\epsilon}{8}\sum_{k=0}^{K-1}\eta_k\\
& \le  \frac{D_h^2}{2} + \Big(M_f^2 + 2\sigma_1^2 + \frac{\Lambda_1(\delta)^2}{2D_h^2K^{\frac{1}{2}}}+\frac{\Lambda_2(\delta)^2}{2D_h^2K}\Big)\Big(\frac{K+2}{2}\Big)^2K\tilde\eta^2  + \Big(\frac{B_1(\delta)}{K^{\frac{1}{2}}} + \frac{B_2(\delta)}{K^{\min\{\frac{1}{\beta},1\}}}\Big)\frac{(K+2)K\tilde\eta}{2} + \frac{\epsilon(K+2)K\tilde\eta}{16},
\end{align*}
where the last inequality is due to $\gamma_0=1$ and $\eta_k=(k+2)\tilde\eta/2$ for all $k\ge0$. Further,  
by \eqref{def:eta-apsg-c-2}, \eqref{ineq:tech-qm}, and rearranging the terms in the above inequality, one has that for all $K\ge2$, it holds that with probability at least $1-\delta$,
\begin{align}
&F(z^K) - F^* \le  \frac{2D_h^2}{(K+2)K\tilde\eta} + \Big(M_f^2 + 2\sigma_1^2 + \frac{\Lambda_1(\delta)^2}{2D_h^2K^{\frac{1}{2}}}+\frac{\Lambda_2(\delta)^2}{2D_h^2K}\Big)(K+2)\tilde\eta + \frac{2B_1(\delta)}{K^{\frac{1}{2}}} + \frac{2B_2(\delta)}{K^{\min\{\frac{1}{\beta},1\}}} + \frac{\epsilon}{4}\nonumber\\
&\overset{\eqref{def:eta-apsg-c-2}}{=} \min_{\hat\eta}\bigg\{\frac{2D_h^2}{(K+2)K \hat\eta} + \Big(M_f^2 + 2\sigma_1^2 + \frac{\Lambda_1(\delta)^2}{2D_h^2K^{\frac{1}{2}}}+\frac{\Lambda_2(\delta)^2}{2D_h^2K}\Big)(K+2)\hat\eta: \hat\eta\in\bigg(0,\frac{1}{4(L_f + L(\epsilon/K))}\bigg]\bigg\}\nonumber \\
&\qquad + \frac{2B_1(\delta)}{K^{\frac{1}{2}}} + \frac{2B_2(\delta)}{K^{\min\{\frac{1}{\beta},1\}}} + \frac{\epsilon}{4}\nonumber\\
&\le \frac{8D_h^2(L_f + L(\epsilon/K))}{(K+2)K} + 2\sqrt{2}D_h\Big(\frac{M_f^2 + 2\sigma_1^2}{K} + \frac{\Lambda_1(\delta)^2}{2D_h^2K^{\frac{3}{2}}}+\frac{\Lambda_2(\delta)^2}{2D_h^2K^2}\Big)^{\frac{1}{2}}  + \frac{2B_1(\delta)}{K^{\frac{1}{2}}} + \frac{2B_2(\delta)}{K^{\min\{\frac{1}{\beta},1\}}} + \frac{\epsilon}{4}\nonumber\\
&\le \frac{8D_h^2L_f}{K^2} + \frac{8D_h^2L(\epsilon)}{K^{\frac{1+3\nu}{1+\nu}}} + \frac{2[2D_h(M_f + \sigma_1) + B_1(\delta)]}{K^{\frac{1}{2}}} + \frac{2 \Lambda_1(\delta)}{K^{\frac{3}{4}}} + \frac{2B_2(\delta)}{K^{\min\{\frac{1}{\beta},1\}}} +\frac{2\Lambda_2(\delta)}{K} + \frac{\epsilon}{4},\label{ineq:fun-val-gap-case3-asgd}
\end{align}
where the second inequality is due to \eqref{ineq:tech-qm} with 
\[
(a,b,c)=\bigg(\frac{2D_h^2}{(K+2)K},\bigg(M_f^2 + 2\sigma_1^2 + \frac{\Lambda_1(\delta)^2}{2D_h^2K^{\frac{1}{2}}}+\frac{\Lambda_2(\delta)^2}{2D_h^2K}\bigg)(K+2), \frac{1}{4(L_f + L(\epsilon/K))}\bigg).
\]
It then follows from \eqref{ineq:fun-val-gap-case3-asgd} that $F(z^K)-F^*\le\epsilon/8+\epsilon/8+\epsilon/8+\epsilon/8+\epsilon/8+\epsilon/8+\epsilon/4=\epsilon$ holds with probability at least $1-\delta$ for all $K$ satisfying \eqref{K4}. Hence, statement (iii) of Theorem \ref{thm:cv-apsg} holds.

\end{proof}

\section{Future work}\label{sec:fw}


{
We conclude by highlighting several directions for future research. First, our analysis relies on input parameters that depend on the maximum number of iterations $K$ and the target accuracy $\epsilon$. It would be interesting to develop strategies for dynamically updating these parameters while retaining universally optimal complexity bounds. Second, it would be worthwhile to extend the techniques for establishing last-iterate convergence, such as those in \cite{liu2024revisiting,shamir2013stochastic}, to our setting. For accelerated methods, another promising direction is to incorporate the techniques developed in \cite{li2026proximal} to establish last-iterate convergence guarantees.
}

\bibliographystyle{abbrv}
\bibliography{ref}

\appendix

\section*{{Appendix}}

\section{{Technical lemmas}}\label{apx:tl}

In this part, we establish several technical lemmas. The following lemma provides a useful inequality for handling the heavy-tailed noise condition.

\begin{lemma}\label{lem:tech-alp}
For any $\alpha\in(1,2]$, it holds that 
\begin{align}\label{ineq:tech-alp}
c \sigma^\alpha \hat{\eta}^{\alpha-1} \le (\alpha-1) c^{\frac{1}{\alpha-1}}\Big(\frac{8}{\varepsilon}\Big)^{\frac{2-\alpha}{\alpha-1}} \sigma^{\frac{\alpha}{\alpha-1}} \hat{\eta} + \frac{\varepsilon}{8} \qquad \forall c,\sigma,\hat{\eta},\varepsilon>0.
\end{align}
\end{lemma}

\begin{proof}
When $\alpha=2$, this inequality holds trivially. We next prove \eqref{ineq:tech-alp} for the case when $\alpha \in(1,2)$. By Young's inequality, one has that $\tau s\le \tau^p/p + s^q/q$ holds for all $\tau, s>0$ and $p,q\ge1$ satisfying $1/p+1/q=1$. Letting $\tau=c \sigma^\alpha \hat{\eta}^{\alpha-1}/s$, $p=1/(\alpha-1)$, and $q=1/(2-\alpha)$, we obtain that 
\begin{align*}
c \sigma^\alpha \hat{\eta}^{\alpha-1} \le \frac{(\alpha-1)c^{1/(\alpha-1)}\sigma^{\alpha/(\alpha-1)}\hat{\eta}}{s^{1/(\alpha-1)}} + (2-\alpha) s^{1/(2-\alpha)}. 
\end{align*}
Further, let $s>0$ be such that $\varepsilon/8=(2-\alpha) s^{1/(2-\alpha)}$. Then, one has $s^{1/(\alpha-1)}=(\varepsilon/(8(2-\alpha)))^{(2-\alpha)/(\alpha-1)}$. Combining these and the above inequality, and using {$(2-\alpha)^{2-\alpha}\le 1$}, we conclude that \eqref{ineq:tech-alp} holds.
\end{proof}

The next lemma will be used to derive the complexity bounds.

\begin{lemma}\label{lem:tech-qm}
Let $a,b,c>0$ be given, $t^*=\min\{c,(a/b)^{1/2}\}$, and $\varphi(t)=a/t + bt$ for $t\in(0,\infty)$. Then, it holds that 
\begin{equation} \label{ineq:tech-qm}
\min_{t\in(0,c]} \varphi(t)=\varphi(t^*)\le a/c + 2(ab)^{1/2}.
\end{equation}
\end{lemma}

\begin{proof}
It is easy to see that the  first relation in \eqref{ineq:tech-qm} holds. We now prove the second relation in \eqref{ineq:tech-qm} by considering two separate cases. If $c\le (a/b)^{1/2}$, one has $\varphi(t^*) = a/c + bc \le a/c + (ab)^{1/2}$. On the other hand,  if $c> (a/b)^{1/2}$, one has $\varphi(t^*) = 2(ab)^{1/2} < a/c + 2(ab)^{1/2}$. Combining these cases, we conclude that the second relation in \eqref{ineq:tech-qm} holds.
\end{proof}

{The following lemma presents a convex comparison technique for deriving concentration inequalities for sub-Weibull random variables, inspired by \cite{marchina2018concentration}. We include a self-contained proof for completeness.

\begin{lemma}\label{lem:cvx-comparison}
Let $\{\xi_k\}$ be a sequence of i.i.d. random variables, and $\{u_k(\cdot)\}$ be a sequence of deterministic functions. Define $u^k = u_k(\xi_{[k]})$ with $\xi_{[k]}=\{\xi_t\}_{t=0}^k$ for all $k=0,1,\ldots$, and let $\E_{\xi_0}[\cdot\,|\, \xi_{[-1]}]:=\E_{\xi_0}[\cdot]$. Suppose that $\E_{\xi_k}[u^k\,|\,\xi_{[k-1]}]=0$ holds for all $k=0,1,\ldots$. Let $\omega$ be a nonnegative integrable random variable, independent of $\{\xi_k\}$, such that
\begin{align}\label{ineq:comparision-w}
\mathbb{P}(|u^k|>v\,|\,\xi_{[k-1]}) \le \mathbb{P}(\omega > v)\qquad\forall v\ge0,k\ge0.
\end{align}
Let $\{\omega_k\}$ be independent copies of $\omega$,
and let $\{s_k\}$ be independent Rademacher random variables. Assume that $\{\omega_k\}$, $\{s_k\}$,
and $\{\xi_k\}$ are mutually independent, and define $\tilde{\omega}_k:=2s_k\omega_k$ for all $k\ge0$. Then, for every $K\ge1$ and any convex continuously differentiable function $h:\R\to\R_+$ satisfying $\E[h(\sum_{k=0}^{K-1}\tilde\omega_k)]<\infty$, we have
\begin{align}\label{ineq:cvx-comp}
    \E\bigg[h\bigg(\sum_{k=0}^{K-1}u^k\bigg)\bigg] \le \E\bigg[h\bigg(\sum_{k=0}^{K-1}\tilde{\omega}_k\bigg)\bigg].
\end{align}
\end{lemma}

\begin{proof}
By the tail bound \eqref{ineq:comparision-w} and conditional
Tonelli's theorem, one has that for all $k\ge0$,
\[
\E_{\xi_k}[|u^k|\mid\xi_{[k-1]}]
=
\int_0^\infty
\mathbb P(|u^k|>v\mid\xi_{[k-1]})\,\mathrm dv
\le
\int_0^\infty \mathbb P(\omega>v)\,\mathrm dv
=
\E[\omega]<\infty.
\]
Taking expectations with respect to $\xi_{[k-1]}$ and applying the
tower property gives $\E[|u^k|]\le\E[\omega]<\infty$. Since $h\ge0$, all expectations involving $h$ below are well defined, possibly with value $+\infty$.

We first establish a one-step comparison. Define $S_0:=0$ and $S_k:=\sum_{i=0}^{k-1}u^i$, $k=1,\ldots,K$. Fix $k\in\{0,\ldots,K-1\}$ and a deterministic $t\in\R$, and let $A:=S_k+t$. Thus, $A$ is integrable and is fixed conditional on $\xi_{[k-1]}$. Let $\hat\xi_k$ be an independent copy of $\xi_k$, independent
of $\xi_{[k]}$, and define $\hat\xi_{[k]}:=\{\xi_0,\ldots,\xi_{k-1},\hat\xi_{k}\}$ and $\hat{u}^k:=u_k(\hat\xi_{[k]})$. Then, conditional on $\xi_{[k-1]}$, $u^k$ and $\hat u^k$ are independent and identically distributed, and moreover, $\E_{\hat\xi_k}[\hat u^k\mid\xi_{[k]}]
=0$. By this, the convexity of $h$, and conditional Jensen's inequality, one has
\[
h(A+u^k) = h(
\E_{\hat\xi_k}[A+u^k-\hat u^k\mid\xi_{[k]}]
) \le \E_{\hat\xi_k}[h(A+u^k-\hat u^k)\mid\xi_{[k]}].
\]
Defining $\varphi_a(v):=[h(a+2v)+h(a-2v)]/2$ for all $a,v\in\R$, taking conditional expectations of the preceding inequality given $\xi_{[k-1]}$, and using the convexity of $h$, we obtain
\begin{align}
\E_{\xi_k}[h(A+u^k)\mid\xi_{[k-1]}]
&\le
\E_{\xi_k,\hat\xi_k}
[h(A+u^k-\hat u^k)\mid\xi_{[k-1]}] \le
\E_{\xi_k,\hat\xi_k}\left[
\frac{h(A+2u^k)+h(A-2\hat u^k)}2
\,\middle|\,\xi_{[k-1]}
\right]
\nonumber\\
&= \E_{\xi_k}\left[
\frac{h(A+2u^k)+h(A-2 u^k)}2
\,\middle|\,\xi_{[k-1]}
\right] =  \E_{\xi_k}[\varphi_A(|u^k|)\mid\xi_{[k-1]}],
\label{ineq:upbd-expec-Sk}
\end{align}
where the first equality holds because $u^k$ and $\hat u^k$ are independent and identically distributed conditional on $\xi_{[k-1]}$. Moreover, notice that $\varphi_a$ is convex, continuously differentiable, and
even. Consequently, $\varphi_a'(v)\ge0$ for all $v\ge0$. By the
fundamental theorem of calculus, one has
\begin{align*}
    \varphi_a(y) - \varphi_a(0) = \int^y_0 \varphi_a^\prime(v)\mathrm{d}v = \int^\infty_0\varphi_a^\prime(v)\mathbbm{1}_{[0,y)}(v)\mathrm{d}v\qquad\forall y\ge0.
\end{align*}
Because $A$ is fixed when conditioning on $\xi_{[k-1]}$, for each realization of $\xi_{[k-1]}$, we let $\varphi_A(\cdot)$ denote the deterministic function obtained by evaluating $\varphi_a(\cdot)$ at $a=A$. Then, conditional Tonelli's
theorem and \eqref{ineq:comparision-w} yield
\begin{align}
    \E_{\xi_k}[\varphi_A(|u^k|)- \varphi_A(0)\,|\,\xi_{[k-1]}] & = \int^\infty_0 \varphi_A^\prime(v)\mathbb{P}(|u^k|>v\,|\,\xi_{[k-1]})\mathrm{d}v \overset{\eqref{ineq:comparision-w}}{\le} \int^\infty_0 \varphi_A^\prime(v)\mathbb{P}(\omega_k>v)\mathrm{d}v\nonumber\\
    & = \E_{\omega_k}[\varphi_A(\omega_{k})- \varphi_A(0)\,|\,\xi_{[k-1]}].\label{ineq:upbd-expec-varphia}
\end{align}
Adding $\varphi_A(0)$ to both sides and taking the expectation with
respect to $s_k$, we obtain
\begin{equation}\label{ineq:upbd-expec-haomega}
\E_{\xi_k}[\varphi_A(|u^k|)\mid\xi_{[k-1]}]
\le\E_{\omega_k}[\varphi_A(\omega_k)\mid\xi_{[k-1]}]
= \E_{\omega_k,s_k}[h(A+2s_k\omega_k)\mid\xi_{[k-1]}] =
\E_{\tilde\omega_k}[h(A+\tilde\omega_k)\mid\xi_{[k-1]}].
\end{equation}
Combining \eqref{ineq:upbd-expec-Sk}-\eqref{ineq:upbd-expec-haomega} and recalling that $A=S_k+t$, we obtain 
\begin{equation}\label{ineq:one-step-comparison}
\E_{\xi_k}[h(S_{k+1}+t)\mid\xi_{[k-1]}]
\le
\E_{\tilde\omega_k}
[h(S_k+\tilde\omega_k+t)\mid\xi_{[k-1]}]
\qquad \forall t\in\mathbb R,
\end{equation}
which completes the proof of the one-step comparison.

To iterate the one-step comparison, define $T_k:=\sum_{j=k}^{K-1}\tilde\omega_j$, $k=0,\ldots,K-1$, and $T_K:=0$. For each fixed $t\in\R$, taking expectations with respect to $\xi_{[k-1]}$ in \eqref{ineq:one-step-comparison} and applying the tower property yield
\begin{align}\label{ineq:one-step-comparison-exp}
    \E[h(S_{k+1}+t)] \le \E[h(S_k + \tilde\omega_k + t)],
\end{align}
where the left-hand expectation is taken with respect to $\xi_0,\ldots,\xi_k$ and the right-hand expectation is taken with respect to $\xi_0,\ldots,\xi_{k-1}$ and $\tilde\omega_k$. Let $\mu_{k+1}$ denote the distribution of $T_{k+1}$. Integrating \eqref{ineq:one-step-comparison-exp} with respect to $\mu_{k+1}(\mathrm dt)$ yields
\begin{align*}
\E[h(S_{k+1}+T_{k+1})] & = \int_{\R}\E[h(S_{k+1}+t)]\,\mu_{k+1}(\mathrm{d}t) \overset{\eqref{ineq:one-step-comparison-exp}}{\le} \int_{\R}\E[h(S_k + \tilde\omega_k + t)]\,\mu_{k+1}(\mathrm{d}t)\\
& = \E[h(S_k + \tilde\omega_k +T_{k+1})] = \E[h(S_k+T_k)].
\end{align*}
Applying this inequality successively for $k=K-1,\ldots,0$, and using $T_K=0$ and $S_0=0$, we conclude that
\[
\E[h(S_K)]
\le
\E[h(S_{K-1}+T_{K-1})]
\le\cdots\le
\E[h(S_0+T_0)]
=
\E[h(T_0)],
\]
which proves \eqref{ineq:cvx-comp} as desired.
\end{proof}

The following technical lemma provides a useful combinatorial bound that will be used in our analysis.

\begin{lemma}\label{lem:frac-bound}
Let $K,m\ge1$ be positive integers and $r\ge-1$ be a real number. Then the following bound holds:\footnote{We use the convention $0!=1$.}
\begin{align}\label{ineq:upper-bd-combin}
    \sum_{\substack{j_0+\cdots+j_{K-1}=m\\j_0,\ldots,j_{K-1}\in\mathbb Z_{+}}}\prod_{k=0}^{K-1}(j_k!)^r \le \bigg[e\bigg(\frac{K}{m} + K^{\max\{-r,0\}}m^r\bigg)\bigg]^{m}.
\end{align}
\end{lemma}

\begin{proof}
For notational convenience, define
\begin{align*}
\mathcal J:=
\left\{
(j_0,\ldots,j_{K-1})\in\mathbb Z_+^K:
\sum_{k=0}^{K-1}j_k=m
\right\},
\qquad
B_r:=\sum_{\boldsymbol j\in\mathcal J}
\prod_{k=0}^{K-1}(j_k!)^r.    
\end{align*}
We consider two cases.

\medskip
\noindent
{Case 1): $r\ge0$.} For $\ell=1,\ldots,\min\{K,m\}$, let 
\begin{align*}
    \mathcal{J}_\ell:=\{(j_0,\ldots,j_{K-1})\in\mathcal{J}:|\{k:j_k>0\}|=\ell\}.
\end{align*}
Choosing the $\ell$ positive coordinates and distributing $m$
among them gives
\begin{align}\label{ineq:upbd-Jell}
|\mathcal J_\ell|
=
\binom K\ell\binom{m-1}{\ell-1},
\end{align}
Indeed, $\binom{K}{\ell}$ counts the ways to choose the $\ell$ positive
coordinates, while $\binom{m-1}{\ell-1}$ counts the ways to distribute $m$
among these coordinates as positive integers. Since $j!\le m^{j-1}$ for any $j=1,\ldots,m$, we have 
\begin{align}\label{ineq:upbd-frac}
\prod_{k=0}^{K-1}(j_k!)^r
=
\prod_{k:j_k>0}(j_k!)^r
\le
m^{r\sum_{k:j_k>0}(j_k-1)}
=
m^{r(m-\ell)}
\qquad \forall (j_0,\ldots,j_{K-1})\in\mathcal J_\ell.    
\end{align}
Combining \eqref{ineq:upbd-Jell} and \eqref{ineq:upbd-frac}, we obtain that
\begin{align*}
B_r &\le \sum_{\ell=1}^{\min\{K,m\}}
\binom K\ell\binom{m-1}{\ell-1}
m^{r(m-\ell)} \le
\sum_{\ell=1}^{m}
\frac{m^\ell}{\ell!}\binom m\ell
\left(\frac Km\right)^\ell
(m^r)^{m-\ell}\\
&\le
e^m\sum_{\ell=0}^{m}
\binom m\ell
\left(\frac Km\right)^\ell
(m^r)^{m-\ell} = \left[e\left(\frac Km+m^r\right)\right]^m,
\end{align*}
where the second inequality holds because $\binom K\ell\le\frac{K^\ell}{\ell!}$ and $\binom{m-1}{\ell-1}=\frac{\ell}{m}\binom m\ell
\le\binom m\ell$, the third inequality follows from $\frac{m^\ell}{\ell!}\le\sum_{i=0}^\infty\frac{m^i}{i!}=e^m$, and the last equality is due to the binomial theorem. This proves \eqref{ineq:upper-bd-combin} as desired.

\medskip
\noindent
Case 2): $-1\le r<0$. Set $q:=-r\in(0,1]$ and $\varphi(j_0,\ldots,j_{K-1}):=\prod_{k=0}^{K-1}(j_k!)^{-1}$. By the concavity of $x\mapsto x^q$ on $[0,\infty)$
and Jensen's inequality,
\begin{equation*}
B_r
=
\sum_{(j_0,\ldots,j_{K-1})\in\mathcal J}\varphi(j_0,\ldots,j_{K-1})^{q}
\le
|\mathcal{J}|^{1-q}
\Bigg(\sum_{(j_0,\ldots,j_{K-1})\in\mathcal J}
\varphi(j_0,\ldots,j_{K-1})\Bigg)^q.
\end{equation*}
Case 1) with $r=0$ and the multinomial theorem give, respectively,
\[
|\mathcal{J}|\le
\left[e\left(\frac Km + 1\right)\right]^m,
\qquad
\sum_{(j_0,\ldots,j_{K-1})\in\mathcal J}\varphi(j_0,\ldots,j_{K-1}) = \frac{K^m}{m!}.
\]
Using these bounds together with $m!\ge(m/e)^m$, we obtain
\[
B_r\le\bigg[e\bigg(\frac Km + 1\bigg)\bigg]^{m(1-q)}\bigg(\frac{eK}{m}\bigg)^{m q} = \bigg[e\bigg(\frac Km + 1\bigg)^{1-q}\bigg(\frac Km\bigg)^q\bigg]^m \le e^m\bigg[\frac Km + \bigg(\frac Km\bigg)^q\bigg]^m = e^m\bigg(\frac Km + K^{-r}m^r\bigg)^m,
\]
where the second inequality uses $(u+1)^{1-q}\le u^{1-q} + 1$ for any $u\ge0$ and
$q\in(0,1]$, and the last equality is due to $r=-q$. This proves \eqref{ineq:upper-bd-combin} for $-1\le r<0$
and completes the proof.

\end{proof}

The following lemma provides an upper bound on the second moment of a sub-Weibull random variable. It will be used to establish high-probability complexity guarantees for our algorithms.

\begin{lemma}\label{lem:var-subweibull}
Let $\beta\in(0,2]$ and $\varsigma>0$, and let $X$ be a sub-Weibull random variable satisfying $\E[\exp\{|X/\varsigma|^\beta\}]\le e$. Then, 
\[
\E[X^2] \le \frac{2e\Gamma(2/\beta)\varsigma^2}{\beta}.
\]
\end{lemma}

\begin{proof}
By Markov's inequality, one has that for any $u\ge0$,
\begin{align*}
\mathbb{P}(|X|\ge u)  = \mathbb{P}\bigg(\exp\bigg\{\bigg|\frac{X}{\varsigma}\bigg|^\beta\bigg\}\ge \exp\Big\{\Big(\frac{u}{\varsigma}\Big)^\beta\Big\}\bigg) \le \frac{\E[\exp\{|X/\varsigma|^\beta\}]}{\exp\{(u/\varsigma)^\beta\}} \le \exp\Big\{1-\Big(\frac{u}{\varsigma}\Big)^\beta\Big\}\qquad\forall u\ge0.    
\end{align*}
It then follows that 
\begin{align*}
\E[X^2] = 2\int^\infty_0 u\mathbb{P}(|X|\ge u)\,\mathrm{d}u \le 2e \int^\infty_0 u\exp\Big\{-\Big(\frac{u}{\varsigma}\Big)^\beta\Big\}\,\mathrm{d}u = \frac{2e\varsigma^2}{\beta}\int^\infty_0 s^{\frac{2}{\beta}-1} e^{-s}\,\mathrm{d}s = \frac{2e\varsigma^2\Gamma(2/\beta)}{\beta},    
\end{align*}
where the first equality uses $y^2=2\int^y_0u\mathrm{d}u=2\int^\infty_0u\mathbbm{1}_{[0,y]}(u)\mathrm{d}u$ for all $y\ge0$, the second equality follows from $s=(u/\varsigma)^\beta$, and the last equality is due to the definition of $\Gamma(\cdot)$. Hence, this lemma holds.
\end{proof}}

\end{document}